\documentclass[nonacm]{acmart}

\usepackage{amsfonts,amsbsy,amsthm}
\usepackage{algorithm, algorithmic, multirow}

\newtheorem{theorem}{Theorem}[section]
\newtheorem{assumption}[theorem]{Assumption}

\newcommand{\E}{\mathbb{E}}
\newcommand{\K}{\mathcal{K}}

\newcommand{\widebar}{\overline}

\AtBeginDocument{%
  \providecommand\BibTeX{{%
    \normalfont B\kern-0.5em{\scshape i\kern-0.25em b}\kern-0.8em\TeX}}}

\setcopyright{acmcopyright}
\acmDOI{XXXXXXX.XXXXXXX}

\acmJournal{TOMACS}

\begin{document}

\title{Contextual Ranking and Selection with Gaussian Processes}

\author{Sait Cakmak}
\orcid{0000-0001-5845-8506}
\affiliation{%
  \institution{Georgia Institute of Technology}
  \department{School of Industrial and Systems Engineering}
  \streetaddress{755 Ferst Drive NW}
  \city{Atlanta}
  \state{Georgia}
  \country{USA}
  \postcode{30332}
}
\email{scakmak3@gatech.edu}

\author{Siyang Gao}
\orcid{0000-0002-3574-6393}
\affiliation{%
  \institution{City University of Hong Kong}
  \department{Department of Advanced Design and Systems Engineering}
  \streetaddress{83 Tat Chee Avenue}
  \city{Kowloon}
  \state{Hong Kong}
  \country{CHINA}
}
\email{siyangao@cityu.edu.hk}

\author{Enlu Zhou}
\orcid{}
\affiliation{%
  \institution{Georgia Institute of Technology}
\department{School of Industrial and Systems Engineering}
  \streetaddress{755 Ferst Drive NW}
  \city{Atlanta}
  \state{Georgia}
  \country{USA}
  \postcode{30332}
}
\email{enlu.zhou@isye.gatech.edu}

\renewcommand{\shortauthors}{Cakmak, Gao, and Zhou}

\begin{abstract}
  In many real world problems, we are faced with the problem of selecting the best among a finite number of alternatives, where the best alternative is determined based on context specific information. In this work, we study the contextual Ranking and Selection problem under a finite-alternative-finite-context setting, where we aim to find the best alternative for each context. We use a separate Gaussian process to model the reward for each alternative, and derive the large deviations rate function for both the expected and worst-case contextual probability of correct selection. 
  We propose the GP-C-OCBA sampling policy, which uses the Gaussian process posterior to iteratively allocate observations to maximize the rate function. We prove its consistency and show that it achieves the optimal convergence rate under the assumption of a non-informative prior.
  Numerical experiments show that our algorithm is highly competitive in terms of sampling efficiency, while having significantly smaller computational overhead.
\end{abstract}



\maketitle

\section{INTRODUCTION}
\label{sec:intro}
Ranking \& Selection (R\&S) studies the problem of identifying the best among a finite number of alternatives, where the true performance of each alternative is only observed through noisy evaluations. The settings of R\&S can be typically categorized into fixed confidence and fixed budget. In the fixed-confidence setting, the goal is to achieve a target probability of correct selection ($PCS$) of the best alternative using as few evaluations as possible, while in the fixed-budget setting one aims to achieve a $PCS$ as high as possible with the given sampling budget.
The R\&S problem has been studied extensively over past few decades, and we refer the reader to \cite{Kim2007RSoverview} and \cite{Chen2015RSoverview} for an overview.

In certain applications, the best alternative may not be the same across the board, and may depend on the underlying {\it context}. The benefit of making context-dependent decisions is easily seen by a simple application of Jensen's inequality: $\E_c[\max_k F(k; c)] \geq \max_k \E_c[F(k; c)]$, where $F(k;c)$ represents the reward of selecting alternative $k$ for the context $c$, and $\E_c[\cdot]$ denotes the expectation with respect to (w.r.t.) $c$. Examples of context-dependent decision making include personalized medicine, where the best drug and dose may depend on the patient's age, gender, and medical history; and recommender systems, where personalized decisions have been the focus of study for over a decade \cite{Nunes2012Recommender}. Context-dependent decision making also arises in R\&S. For example, based on a set of forecasted market conditions (contexts), we can identify a set of alternative configurations of a complex manufacturing system, which can be simulated under any given context to determine the most profitable configuration to use when the market conditions are realized.

In this work, we study the contextual R\&S problem, in which the rewards are a function of the context.
Our goal is to identify the best alternative for each context under a fixed budget.
Much like the classical R\&S problem, at each iteration, the decision maker selects an alternative-context pair to evaluate and observes a noisy evaluation of the true reward. With a finite sampling budget and noisy observations,
it is not possible to identify the best alternative with certainty, and
we need to design a sampling policy, which takes in the current estimate of rewards and outputs the next alternative-context pair to sample, in order to achieve
the highest possible ``aggregated" $PCS$
when the budget is exhausted.   
The aggregation is needed because in the contextual R\&S, for any sampling policy, $PCS$ is also context dependent, i.e., for each context $c$ there is a  $PCS(c)$.  This defines multiple objectives to consider when designing the sampling policy.
In this work, we consider two approaches to aggregate $PCS(c)$'s to a scalar objective. The first one is the expected $PCS$ \cite{Gao2019Covariates,shen2021ranking}, denoted as $PCS_E$, which is the expectation or the weighted average of $PCS(c)$ given a set of normalized weights, and the second one is the worst-case $PCS$ \cite{li2020contextdependent}, denoted as $PCS_M$, which is the minimum $PCS(c)$ obtained across all contexts.

The contextual R\&S problem has seen an increasing interest in past few years. Notable works that study this problem under finite alternative-context setting include but not limited to \cite{Gao2019Covariates, Jin2019Analytics, li2020contextdependent, shen2021ranking}. \cite{li2020contextdependent} focuses on worst-case $PCS$, uses independent normal random variables to model rewards, and proposes a one-step look-ahead policy with an efficient value function approximation scheme to maximize $PCS_M$. \cite{shen2021ranking} assumes that the reward for each alternative is a linear function of the context and proposes a two-stage algorithm based on the indifference zone formulation for optimizing the expected $PCS$. The most closely related work to ours is \cite{Gao2019Covariates}. They model the rewards using independent normal random variables, extend the analysis in \cite{Glynn2004LargeDev} to derive the large deviations rate function for the contextual $PCS$, and propose an algorithm that obtains the asymptotically optimal allocation ratio for both the worst-case and expected $PCS$. 
In contrast, we use a Gaussian process, model correlations between contexts, and extend the large deviations approach to leverage the efficient inference brought in by our statistical model.
\cite{Jin2019Analytics} also follows a large deviations approach similar to that of \cite{Gao2019Covariates}; however, their algorithm does not perform well when only the observed data is used to make decisions.

In this work, we use a separate Gaussian process (GP) to model the reward function for each alternative. By leveraging the hidden correlation structure within the reward function, GPs offer significant improvements in posterior inference over independent normal random variables, which are commonly used in the R\&S literature. Due to the finite solution space we focus on, when compared to a simpler multi-variate Gaussian prior, GPs may appear to complicate things by introducing kernels, which are typically used in continuous spaces. We prefer GPs since they have hyper-parameters that can be trained to better fit the observations as the optimization progresses. In contrast, a multi-variate Gaussian prior is a static object that needs to be specified beforehand based on limited domain knowledge.
The contributions of our paper are summarized as follows:
\begin{itemize}
    \item Using the posterior mean of the GP as the predictor of the true rewards and leveraging a novel decomposition of the GP update formulas, we derive the large deviations rate function for the contextual PCS, and show that it is identical for both $PCS_E$ and $PCS_M$.
    
    \item We propose a sequential sampling policy that aims to maximize the rate function, and uses the GP posterior mean and variance to select the next alternative-context to sample. 
    Our sampling policy, GP-C-OCBA, is based on the same idealized policy as the C-OCBA policy \cite{Gao2019Covariates}, and mainly differs in the statistical model and the predictors used. 
    
    \item Under a set of mild assumptions, we show that GP-C-OCBA almost surely identifies the best alternative for each context as the sampling budget goes to infinity. Under the additional assumption of independent priors, we show that the sampling allocations produced by GP-C-OCBA converge to the optimal static allocation ratio.
    
    \item We build on this to establish, to the best of our knowledge, the first convergence rate result for a contextual R\&S procedure in the literature, We show that the resulting contextual PCS converges to $1$ at the optimal exponential rate. 
    
    \item Finally, we demonstrate the performance of GP-C-OCBA in numerical experiments. We see that GP-C-OCBA achieves significantly improved sampling efficiency when compared with C-OCBA \cite{Gao2019Covariates}, DSCO \cite{li2020contextdependent}, TS, and TS+ \cite{shen2021ranking}; and is highly competitive against the integrated knowledge gradient (IKG) algorithm \cite{PEARCE2018IKG-REVI}, which uses the same GP model but requires significantly larger computational effort to decide on the next alternative-context to evaluate.
\end{itemize}

\section{PROBLEM FORMULATION} \label{section-formulation}
We consider a finite set of alternatives $\{ k \in \K \}$ and a finite set of contexts $\{c \in \mathcal{C} \subset \mathbb{R}^d \}$, where $d$ is the dimension of the context variable. 
We assume that $\K$ is a set of categorical inputs, i.e., that there is no metric defined over $\K$. 

The observations of the reward function, $F(k, c)$, is subject to noise, and the decision maker is given a total budget of $B$ function evaluations. The aim is to identify the true best alternative for each context,
\begin{equation*}
    \pi^*(c) := {\arg \max}_{k \in \mathcal{K}} F(k, c),
\end{equation*}
so that a decision can be made immediately once the final context is revealed. Since $F(k, c)$ is unknown and the observations are noisy, $\pi^*(c)$ cannot be identified with certainty using a finite budget. If we let $\mu_B(k, c)$ denote the estimate of $F(k, c)$ obtained using $B$ function evaluations, we can write the corresponding estimated best alternative as
\begin{equation*}
    \pi^B(c) := {\arg \max}_{k \in \mathcal{K}} \mu_B(k, c).
\end{equation*}
Note that $\pi^B(c)$ is determined by the outcome of the $B$ function evaluations, as well as the sampling policy used to allocate those $B$ evaluations. Thus, for any given sampling policy, $\pi^B(c)$ is a random variable. We can measure the quality of the given sampling policy, for any context c, by the corresponding probability of correct selection ($PCS$) after exhausting the budget $B$:
\begin{equation*}
    PCS^B(c) = P(\pi^{B}(c) = \pi^*(c)).
\end{equation*}

As discussed in the introduction, $PCS^B(c)$ for all contexts define multiple objectives to consider while designing a sampling policy, with each $PCS$ becoming larger as we allocate more samples to the corresponding $c$. Since we have a total sampling budget $B$ rather than individual budgets for each context, it makes sense to work with a scalar objective instead.
In the literature, there are two common approaches for constructing a scalar objective from $[PCS^B(c)]_{c \in \mathcal{C}}$. The first approach assumes that we are given a set of normalized weights $w(c)$ for each $c \in \mathcal{C}$ or the context variable follows a probability distribution $\{w(c), c\in\mathcal{C}\}$, and uses the expected $PCS$ \cite{Gao2019Covariates}
\begin{equation*}
    PCS^B_E = \E_{c \sim w(c)}[PCS^B(c)]
\end{equation*}
as the objective to be maximized. The other alternative takes a worst-case approach and aims to maximize the worst-case $PCS$ \cite{li2020contextdependent}
\begin{equation*}
    PCS^B_M = \min_{c \in \mathcal{C}} PCS^B(c).
\end{equation*}

We refer to either of $PCS^B_E$ and $PCS^B_M$ as the contextual $PCS$, and aim to design a sampling policy that maximizes the contextual $PCS$ with the given sampling budget $B$.
We propose an iterative approach that repeats the following steps at each iteration until the sampling budget is exhausted.
\begin{itemize}
    \item Use the reward observations collected so far to update the statistical model of the reward function.
    \item With the objective of maximizing the large deviations rate function, use the sampling policy to decide on next alternative-context, from which to sample one more observation. 
\end{itemize}
In the following sections, we introduce our statistical model, which builds on a Gaussian process (GP) that leverages the hidden correlation structure in the reward function for more efficient posterior inference, derive the large deviations rate function using the posterior mean of the GP as the predictor of the rewards, and introduce our sampling policy, which aims to maximize the large deviations rate function.

\section{STATISTICAL MODEL} \label{section-statistical-model}
Gaussian processes are a class of non-parametric Bayesian models that are highly flexible for modeling continuous functions.
By restricting to a discrete subset of the solution space, they also provide a powerful alternative to a multi-variate Gaussian prior for modeling a discrete set of designs. 
In this work, we use Gaussian processes to model the reward function. We use a separate GP for each alternative, which allows modeling the correlations between the reward function corresponding to each context for that alternative. 

Let $\mathcal{F}_n(k) = \{ D_n(k), Y_n(k) \}$ denote the history of observations corresponding to alternative $k$ up to time $n$ (i.e., $n$ total observations across all alternatives), where $D_n(k)$ and $Y_n(k)$ denote the set of contexts that have been evaluated and the corresponding observations, respectively. Given the history $\mathcal{F}_n(k)$ and a set of hyper-parameters $\theta$, the GP implies a multi-variate Gaussian posterior distribution on the reward function, given by:
\begin{equation*}
    F(k, \mathcal{C}) \mid \mathcal{F}_n, \theta \sim \mathcal{N}(\mu_n(k, \mathcal{C}), \Sigma_n(\mathcal{C}, \mathcal{C}; k));
\end{equation*}
where $\mu_n(k, \mathcal{C})$ and $\Sigma_n(\mathcal{C}, \mathcal{C}; k)$ are the posterior mean vector and covariance matrix, which, assuming Gaussian observation noise, can be computed in closed form as
\begin{equation*}
    \mu_n(k, \mathcal{C}) = \mu_0(k, \mathcal{C}) + \Sigma_0(\mathcal{C}, D_n(k); k) A_n^{-1}(k) (Y_n(k) - \mu_0(D_n(k)))^{\top},
\end{equation*}
\begin{equation*}
    \Sigma_n(\mathcal{C}, \mathcal{C}; k) = \Sigma_0(\mathcal{C}, \mathcal{C}; k) - \Sigma_0(\mathcal{C}, D_n(k); k) A_n^{-1}(k) \Sigma_0(D_n(k), \mathcal{C}; k),
\end{equation*}
with $A_n(k) = \Sigma_0(D_n(k), D_n(k); k) + diag(\sigma^2(k, D_n(k)))$, where $\sigma^2(\cdot, \cdot)$ denotes the known standard deviation of the Gaussian observation noise. 
The prior mean, $\mu_0$, is commonly set to a constant, e.g., $\mu_0(\cdot, \cdot) = 0$, and the prior covariance, $\Sigma_0$, is commonly chosen from popular covariance kernels such as squared exponential,
    $\Sigma_0(c, c'; k \mid \theta) = \theta_0 \exp \left( - \frac{1}{2} dist^2 \right)$, 
and Mat\`ern,
\begin{equation*}
    \Sigma_0(c, c'; k \mid \theta) = \theta_0 \frac{2^{1-\nu}}{\Gamma(\nu)}\left(\sqrt{2 \nu} dist \right)^{\nu} K_{\nu}\left(\sqrt{2 \nu} dist \right),
\end{equation*}
where $dist = \sqrt{\langle \theta_{1:d} (c - c'), c - c' \rangle }$ is the coordinate-wise scaled Euclidean distance,
$\Gamma(\cdot)$ is the gamma function, $K_{\nu}(\cdot)$ is the modified Bessel function of second kind, $\nu$ is a shape parameter that is commonly set to $\nu = 5/2$.
$\theta_0$ and $\theta_{1:d}$ denote the output-scale and the length-scale parameters, respectively, and they are collectively referred to as the hyper-parameters of the GP. Throughout, we keep the dependence on the hyper-parameters $\theta = \{\theta_0, \theta_{1:d}\}$ implicit in the notation.
The observation noise level, $\sigma^2(\cdot, \cdot)$, is commonly unknown and gets replaced with a plug-in estimate, which is
optimized jointly with the hyper-parameters $\theta$, using maximum likelihood or maximum a-posteriori estimation. In addition to $\mu_n(k, c)$ and $\Sigma_n(c, c'; k)$ to denote the posterior mean and covariance, we use $\Sigma_n(k, c) := \Sigma_n(c, c; k)$ to denote the posterior variance.

In contrast to the common practice of using a single GP defined over the whole design space (i.e., the product space of alternative and contexts),
we model the rewards for each alternative using an independent GP model, which is defined over the context space $\mathcal{C}$ and trained using only the observations corresponding to that alternative. There are a couple of reasons for using an independent GP model for each alternative. 
\begin{itemize}
    \item With $\K$ as a set of categorical inputs, we do not have a metric defined over $\K$. Thus, we cannot use a covariance kernel with the categorical alternative values as the inputs. It is possible to define a latent embedding of $\K$ into a Euclidean space and apply a covariance kernel in the embedded space (see, e.g., \cite{Guo2016Embedding,Feng2020ContextualBO}). However, this introduces many additional hyper-parameters to the model, resulting in a non-convex optimization problem with many local optima for training the model. We found the predictive performance of such models to be highly sensitive to initial values of these hyper-parameters.
    
    \item The complexity of the GP inference is dominated by the inversion of matrix $A_n$, which has a $\mathcal{O}(n^3)$ cost for an $n \times n$ matrix using standard techniques. If we use use a single GP fit on all $n$ observations, this results in an $n \times n$ matrix $A_n$. In contrast, when using $K := |\K|$ independent GP models, each with $N_n(k)$ training inputs, we have $K$ matrices $A_n(k)$, each of size $N_n(k) \times N_n(k)$, with $\sum_{k} N_n(k) = n$. The GP inference with independent models has a total cost of $\mathcal{O}(\sum_k n_k^3)$. If we assume that the samples are evenly distributed across alternatives, i.e., $N_n(k) = n / K$, this results in a $\mathcal{O}(n^3 / K^2)$ cost of inference for $K$ independent models, which is much cheaper than $\mathcal{O}(n^3)$ for a single GP model.
\end{itemize}
On the other hand, if the set of alternatives belongs to a metric space, using a single GP model with a well defined covariance kernel over alternatives could lead to better sampling efficiency, and may be preferable when the samples are expensive or the sampling budget is severely limited. Although our derivation utilizes the independence of the models across different alternatives, the resulting GP-C-OCBA algorithm presented in this work is agnostic to the specifics of the GP model, and it can be used with either a single GP defined over the alternative-context space or a GP model with the latent embedding as discussed in the first point, whenever such models are found to be appropriate.


\begin{figure}[ht]
    \centering
    \includegraphics[width=\textwidth]{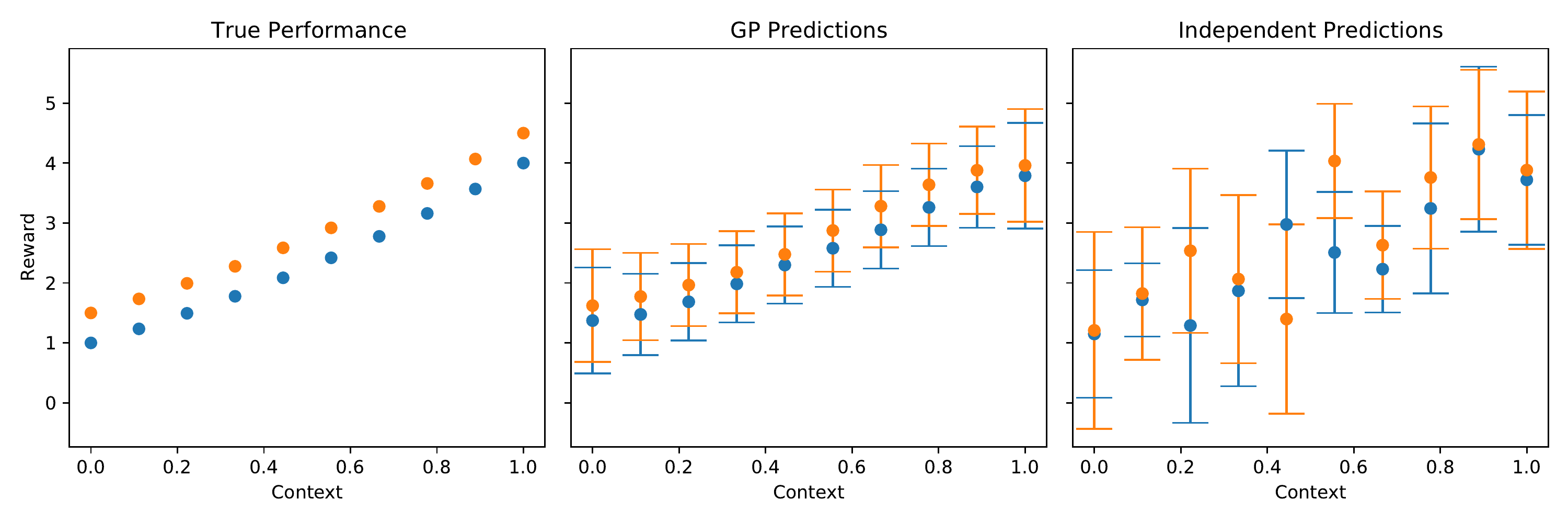}
    \caption{Comparison of the GP predictions with those of an independent frequentist model. Left to right, the plots show the true performances of the two alternatives (y-axis) as a function of the contexts (x-axis), the GP predictions and the independent predictions after sampling each alternative-context pair $10$ times. The error bars denote the two standard deviations of the predictive variance. It is seen that the GP model smooths out the errors and leads to much more reliable predictions and lower predictive variance.}
    \label{fig:gp_example}
\end{figure}

In Figure \ref{fig:gp_example}, we see a comparison of the predictions of the GP model used in this paper, and the independent model used by \cite{Gao2019Covariates}. Both models are fitted on identical data, and the posterior mean is plotted along with error bars denoting two standard deviations of the predictive variance. The GP model smooths out the errors in the observations and leads to significantly more accurate predictions. In addition, we see that the error bars for the GP predictions are significantly smaller, which is attributed to the model using neighboring observations to lower the predictive uncertainty.

\section{THE LARGE DEVIATIONS RATE FUNCTION} \label{section-derivation}
In this section, we derive the large deviations rate function corresponding to the contextual $PCS$ measures. In particular, we calculate the rate at which the probability of false selection,
\begin{equation*}
    PFS^n_E = 1 - PCS^n_E \text{ or } PFS^n_M = 1 - PCS^n_M,
\end{equation*}
approaches zero as the number of samples $n \rightarrow \infty$. The main result is summarized in the following theorem.

\begin{theorem}\label{thm-pcs-rate-function}
    Suppose that the observations are given as $y^n(k, c) = F(k, c) + \epsilon^n(k, c)$ where $\epsilon^n(k, c) \sim \mathcal{N}(0, \sigma^2(k, c))$ and $\epsilon^n(k, c)$ are independent across $n, k, c$;
    and the best alternative, $\pi^*(c)$, is unique for all $c \in \mathcal{C}$. 
    Using $\mu_n(k, c)$ to predict the rewards, the large deviations rate function for both $PCS^n_E$ and $PCS^n_M$ is given by
    \begin{equation} \label{eq-final-rate-function}
        \lim_{n \rightarrow \infty} \frac{1}{n} \log PFS^n_{\sim} = - \min_{c \in \mathcal{C}} \min_{k \neq \pi^*(c)} \frac{(F({\pi^*(c)}, c) - F(k, c))^2}{2(\sigma^2({\pi^*(c)}, c) / p(\pi^*(c), c) + \sigma^2(k, c) / p(k, c))},
    \end{equation}
    where $PFS^n_{\sim}$ is either of $PFS^n_E$ or $PFS^n_M$, and $p(k, c)$ denotes the fraction of samples allocated to $k, c$.
\end{theorem}

The derivation of the result follows the ideas originating in \cite{Glynn2004LargeDev}, with modifications to accommodate the use of the posterior mean $\mu_n(k, c)$ instead of the sample mean.
We first show that the $PFS^n(c)$ can be upper and lower bounded by a constant multiple of $\max_{k \neq \pi^*(c)} P(\mu_n(\pi^*(c), c) < \mu_n(k, c))$. Then, if we can identify a rate function for $P(\mu_n(\pi^*(c), c) < \mu_n(k, c))$, we can extend this with a minimum over the contexts and alternatives to find the rate function of the contextual $PFS$. The rest of the derivation is focused on analyzing the log the moment generating function of $\mu_n(k, c)$, which requires a novel decomposition of the GP update formulas, and showing that it is asymptotically equivalent to that of the sample mean.
This connects us back to the derivation in \cite{Glynn2004LargeDev}, and we follow the steps therein to obtain the final result. The full proof is presented in the appendix.

It is worth remarking that the large deviations rate function presented in Theorem \ref{thm-pcs-rate-function} is identical to the rate function derived in \cite{Gao2019Covariates} using the sample mean estimator, which is the same as the rate function originally derived in \cite{Glynn2004LargeDev} with an additional minimum over the contexts.
Our analysis shows that the same rate function is still applicable when the independent Gaussian distribution is replaced with a Gaussian process. 
This is not that surprising since, as we see in the analysis,
the effects of correlations in the model disappear as more observations are added, and two estimators are asymptotically equivalent.
As a result, using a GP enables efficient inference by learning the similarities between contexts, resulting in significantly improved performance with small sampling budgets, while retaining similar asymptotical properties.

\section{SAMPLING POLICY} \label{section-sampling-policy}
In this section, we introduce the GP-C-OCBA policy, which aims to maximize the rate function presented in Theorem \ref{thm-pcs-rate-function}, adapt the IKG policy from the literature to our problem setting, and present a comparison of the computational cost of the two policies.

\subsection{GP-C-OCBA}
In classical R\&S literature, optimal computing budget allocation (OCBA, \cite{Chen2015RSoverview}) is a popular approach for maximizing the $PCS$ asymptotically. 
For the contextual R\&S problem,
\cite{Gao2019Covariates} derives the Karush-Kuhn-Tucker (KKT) conditions for maximizing the rate function (\ref{eq-final-rate-function}), and proposes an idealized sampling policy that iteratively realizes the KKT conditions.
The idealized policy derived by \cite{Gao2019Covariates} (see Section III.C of the paper for derivation) relies on $F(k, c)$, $\sigma^2(k, c)$, and $\pi^*(c)$, which are not known in practice, as well as $\hat{p}(k, c)$, which denotes the fraction of total samples allocated to $(k, c)$ so far and is different than $p(k, c)$ used in the derivation to denote the idealized asymptotic allocation rate.
For a practical algorithm, \cite{Gao2019Covariates} replaces $F(k, c)$ and $\sigma^2(k, c)$ with the sample mean and variance, respectively, and $\pi^*(c)$ with the corresponding estimate to define the C-OCBA policy.

As shown in Theorem \ref{thm-pcs-rate-function}, using the GP model yields the same rate function as using the sample mean predictors in \cite{Gao2019Covariates}, so we take a similar approach and use the posterior mean $\mu_n(k, c)$ and the posterior variance $\Sigma_n(k, c)$, which are our estimates based on $n$ observations collected so far, in place of $F(k, c)$ and $\sigma^2(k, c) / \hat{p}(k, c)$ in the idealized policy. The resulting GP-C-OCBA policy is presented in Algorithm \ref{alg:gp-c-ocba}. 

\begin{algorithm}[th]
  \caption{GP-C-OCBA for Contextual R\&S}
  \label{alg:gp-c-ocba}
\begin{algorithmic}[1]
  \STATE Use a pre-determined rule to allocate initial samples to each alternative. Let $N_0$ be the total number of initial observations used.
  \FOR{$n = N_0, \ldots, B - 1$}
    \STATE Update the GP model with the available observations, calculate $\mu_n(k, c), \Sigma_n(k, c), \pi^n(c)$.
    \STATE For all $c \in \mathcal{C}$ and $k \neq \pi^n(c)$, calculate
    \begin{equation*}
        \zeta(k, c) = \frac{(\mu_n(\pi^n(c), c) - \mu_n(k, c))^2}{\Sigma_n(\pi^n(c), c) + \Sigma_n(k, c)};
    \end{equation*}
    and set 
    \begin{equation*}
        \psi^{(1)}(c) = \frac{\hat{p}(\pi^n(c), c)}{\Sigma_n(\pi^n(c), c)}; \qquad \psi^{(2)}(c) = \sum_{k \neq \pi^n(c)} \frac{\hat{p}(k, c)}{\Sigma_n(k, c)}.
    \end{equation*}
    \STATE Solve for $(\Tilde{k}^*, \Tilde{c}^*) = {\arg \min}_{k \neq \pi^n(c), c \in \mathcal{C}} \zeta(k, c)$, and draw an additional sample from $(\pi^n(\Tilde{c}^*), \Tilde{c}^*)$, if $\psi^{(1)}(\Tilde{c}^*) < \psi^{(2)}(\Tilde{c}^*)$, and draw an additional sample from $(\Tilde{k}^*, \Tilde{c}^*)$ otherwise.
  \ENDFOR
  \STATE {\bfseries Return:} $\pi^B(c) = \arg \max_k \mu_{B}(k, c), c \in \mathcal{C}$ as the set of predicted best alternatives.
\end{algorithmic}
\end{algorithm}

Our experiments (in Section \ref{section-numerical-experiments}) show that using the GP model with the GP-C-OCBA sampling strategy leads to significantly higher contextual $PCS$ using the same sampling budget, thanks to the improvements in the posterior inference from using a statistical model that leverages the hidden correlation structure in the reward function.
An additional benefit of our approach over \cite{Gao2019Covariates} is in its applicability when the initial sampling budget is too small to draw multiple samples from each alternative. Using normal random variables to model each alternative-context pair requires a small number of samples from each pair for the initial estimate of the variance, which may limit the applicability of the algorithm when the sampling budget is limited. The GP prior, on the other hand, can be trained using very few samples for each alternative, rather than each alternative-context pair, thus the modified algorithm can be used even with a limited sampling budget.

A final remark about the GP-C-OCBA is that in the extreme case where there's only a single context, the sampling policy is equivalent to the OCBA-2 algorithm presented in \cite{li-gao2021convergence}, which has been shown to achieve the optimal convergence rate in the standard R\&S setting. Thus, GP-C-OCBA can be viewed as a principled extension of a rate optimal R\&S algorithm to the contextual R\&S setting.

\subsection{Integrated Knowledge Gradient} \label{section-ikg}
On a related note, another applicable method for the contextual R\&S problem is the integrated knowledge gradient (IKG) algorithm, which has been developed for the closely related problem of contextual Bayesian optimization. In this section, we adapt the IKG algorithm to our setting, and compare it with GP-C-OCBA. IKG offers a strong benchmark for our method, since it is based on the same GP model and has demonstrated superior sampling efficiency in prior work \cite{PEARCE2018IKG-REVI, pearce2020IKG-ConBO}.

Knowledge Gradient (KG) \cite{Frazier2009KG} is a value-of-information type policy that was originally proposed for the R\&S problem and later expanded to global optimization of black-box functions. It is well known for its superior sampling efficiency, which comes at a significant computational cost. For a given context $c'$, we can write the KG factor, which measures the expected improvement in value of the maximizer for context $c'$ from adding an additional sample at $(k, c)$, as 
\begin{equation*}
    \text{KG}(k, c; c') = \E_n[\max_{k' \in \K} \mu_{n+1}(k', c') \mid (k^{n+1}, c^{n+1}) = (k, c)] - \max_{k' \in \K} \mu_n(k', c').
\end{equation*}
In the classical R\&S setting, where $c$ and $c'$ are redundant (i.e., there is only a single context), the KG policy operates by evaluating the alternative $k^* = \arg \max_k \text{KG}(k, c; c)$. To extend this to the contextual Bayesian optimization problem, \cite{PEARCE2018IKG-REVI,zhang2020IKG,pearce2020IKG-ConBO} each study an integrated (or summed) version of KG, under slightly different problem settings, where either the context space or both alternative-context spaces are continuous. The main differences between these three works are in how they approximate and optimize the integrated KG factor in their respective problem settings.
For our problem setting, these approaches are equivalent, and we refer to the sampling policy as IKG. We use IKG as a benchmark to evaluate the sampling efficiency of our proposed algorithm. 

The IKG factor is simply a weighted sum of KG factors corresponding to each context. It measures the weighted sum of the improvement in value of maximizers, and is written as
\begin{equation*}
    \text{IKG}(k, c) = \sum_{c' \in \mathcal{C}} \text{KG}(k, c; c') w(c').
\end{equation*}
At each iteration, the IKG policy samples the alternative-context pair that maximizes the IKG factor, $(\tilde{k}^*, \tilde{c}^*) = \arg \max_{k \in \K, c \in \mathcal{C}} \text{IKG}(k, c)$. The IKG policy is summarized in Algorithm \ref{alg:IKG}.

\begin{algorithm}[th]
  \caption{IKG for Contextual R\&S}
  \label{alg:IKG}
\begin{algorithmic}[1]
  \STATE Use a pre-determined rule to allocate initial samples to each alternative. Let $N_0$ be the total number of initial observations used.
  \FOR{$n = N_0, \ldots, B - 1$}
    \STATE Update the GP model with the available observations, calculate $\mu_n(k, c)$ and $\Sigma_n(c, c'; k)$.
    \STATE For all $c, c' \in \mathcal{C}$ and $k \in \mathcal{K}$, use Algorithm 1 of \cite{PEARCE2018IKG-REVI} to calculate $\text{KG}(k, c; c')$ and set
    \begin{equation*}
        \text{IKG}(k, c) = \sum_{c' \in \mathcal{C}} \text{KG}(k, c; c') w(c').
    \end{equation*}
    \STATE Solve for $(\Tilde{k}^*, \Tilde{c}^*) = {\arg \max}_{k \in \mathcal{K}, c \in \mathcal{C}} \text{IKG}(k, c)$, and draw an additional sample from $(\Tilde{k}^*, \Tilde{c}^*)$.
  \ENDFOR
  \STATE {\bfseries Return:} $\pi^B(c) = \arg \max_k \mu_{B}(k, c), c \in \mathcal{C}$ as the set of predicted best alternatives.
\end{algorithmic}
\end{algorithm}

The main difficulty with using the IKG policy is its computational cost. In the finite alternative-context setting that we are working with, the $\text{KG}(k; c)$ can be computed exactly using Algorithm 1 from \cite{PEARCE2018IKG-REVI}, which has a cost of $\mathcal{O}(K \log K)$ for any pair $(k, c)$, given $\mu_n(\cdot, \cdot)$ and $\Sigma_n(\cdot, \cdot)$. This translates to an $\mathcal{O}(|\mathcal{C}| K \log K)$ cost for calculating the IKG factor for a given $(k, c)$.
In total, to find the next pair to sample using IKG costs $\mathcal{O}(|\mathcal{C}|^2 K^2 \log K)$ for calculating the IKG factors, and an additional $\mathcal{O}(\sum_k [n_k^3 + |\mathcal{C}|^2 n_k + |\mathcal{C}|n_k^2])$ to calculate posterior mean and covariance matrices, where $n_k$ denotes the total number of samples allocated to alternative $k$.

On the other hand, the cost of GP-C-OCBA is dominated by the cost of calculating the posterior mean and variance for each alternative-context pair, which has a total cost of $\mathcal{O}(\sum_k [n_k^3 + |\mathcal{C}| n_k^2])$. Note that we avoid the $\sum_k |\mathcal{C}|^2 n_k$ term since our algorithm only requires the posterior variance, as well as the $\mathcal{O}(|\mathcal{C}|^2 K^2 \log K)$ cost of IKG calculations.
This puts GP-C-OCBA at a significant advantage in terms of computational complexity.

\section{CONVERGENCE ANALYSIS}
In this section, we analyze the convergence properties of the GP-C-OCBA algorithm, and show that it identifies $\pi^*(c)$ almost surely as the simulation budget $B \rightarrow \infty$.
In addition, under the assumption of an independent prior distribution, we show that the allocation ratio produced by GP-C-OCBA converges almost surely to the optimal allocation ratio obtained by maximizing the rate function in Theorem \ref{thm-pcs-rate-function}. Finally, we build on this result to show that the contextual PFS converges to zero at the optimal exponential rate.

\subsection{Consistency of GP-C-OCBA}
We prove the result under a set of mild assumptions.
\begin{assumption} \label{assumption-consistency} \
\begin{enumerate}
    \item The best alternative, $\pi^*(c)$, is unique for all $c \in \mathcal{C}$.
    
    \item The observation noise is normally distributed with known variance, i.e., $\epsilon^n(k, c) \sim \mathcal{N}(0, \sigma^2(k, c))$ with known $\sigma^2(k, c)$, and $\epsilon^n(k, c)$ are independent across $n, k, c$.
    
    \item The GP prior is fixed across iterations.
    
    \item The prior correlation coefficient between any two contexts for any given alternative, \linebreak $Corr_0(c, c'; k) := \frac{\Sigma_0(c, c'; k)}{\sqrt{\Sigma_0(k, c) \Sigma_0(k, c')}}$, is strictly between $[-1, 1]$, i.e., $-1 < Corr_0(c, c'; k) < 1$. In other words, $\Sigma_0(\mathcal{C}, \mathcal{C}; k)$ is a strictly positive definite matrix.
\end{enumerate}
\end{assumption}

The first three points are common assumptions from the literature that simplify the analysis. The known observation noise level is needed for exact GP inference, 
and the assumption of a fixed prior eliminates the need to study the hyper-parameters in the analysis.
In practice, the form of the GP prior is kept fixed, however, the hyper-parameters of the GP are periodically updated to better fit the data, which is typically done via maximum likelihood estimation.
Though showing this rigorously for an adaptive sampling policy is a daunting task, the hyper-parameters of the GP prior typically converge to some final value and remain fixed after that. 
The last point is merely technical, since in practice it would not make sense to consider two contexts that are assumed to be perfectly correlated. The following lemma shows that this implies that $\Sigma_n(\mathcal{C}, \mathcal{C}; k)$ must remain strictly positive definite.
Note that since the posterior variance is a deterministic function of the sampling decisions (i.e., it is independent of the observations) the following statement holds for any sequence of sampling decisions.

\begin{lemma} \label{lemma-posterior-correlation}
    If $-1 < Corr_0(c, c'; k) < 1$, then $-1 < Corr_n(c, c'; k) < 1, \forall n \geq 0$.
\end{lemma}

The proofs of this and the following statements are included in the appendix. The posterior variance of a GP model at a given point decreases monotonically as we add more samples to the model. Thus, for a given alternative-context, the posterior variance in the limit is $0$ if that pair gets sampled infinitely often, and is a strictly positive value otherwise. The lemma below shows this rigorously.

\begin{lemma} \label{lemma-posterior-variance}
    Under Assumption \ref{assumption-consistency}, $\Sigma_n(k, c) \rightarrow 0$ almost surely if and only if $N_n(k, c) \rightarrow \infty$. 
\end{lemma}

The intuitive interpretation of Lemma \ref{lemma-posterior-variance} is that, despite modeling correlations between contexts, we cannot drive the uncertainty about the reward function of an alternative-context to zero, unless we collect infinitely many observations for that pair. This rules out the pathological cases where the statistical model treats a reward estimate as certain purely based on observations of other alternative-context pairs. 

We will build on these results to show that GP-C-OCBA policy samples each alternative-context pair infinitely often. We will first establish that if a given context gets sampled infinitely often, all alternative-context pairs corresponding to that context must get sampled infinitely often. 

\begin{lemma} \label{lemma-single-context-infinite-samples}
    Under Assumption \ref{assumption-consistency}, using the GP-C-OCBA policy, for any $c \in \mathcal{C}$,
    \begin{equation*}
        \lim_{n \rightarrow \infty} \sum_{k \in \K} N_n(k, c) = \infty \implies \lim_{n \rightarrow \infty} N_n(k, c) = \infty, \forall k \in \K, \text{ almost surely.}
    \end{equation*}
\end{lemma}

If we only had a single context, Lemma \ref{lemma-single-context-infinite-samples} would be sufficient to prove the consistency of the algorithm. However, with multiple contexts, we need to ensure that a strict subset of contexts cannot consume all observations after a certain iteration. The following theorem extends Lemma \ref{lemma-single-context-infinite-samples} to show that, as the budget goes to infinity, all alternative-context pairs are sampled infinitely often by GP-C-OCBA.

\begin{theorem} \label{theorem-infinitely-sample}
    Under Assumption \ref{assumption-consistency}, the GP-C-OCBA policy samples each alternative infinitely often, i.e.,
    \begin{equation}
        \lim_{n \rightarrow \infty} N_n(k, c) = \infty, \forall k \in \K, c \in \mathcal{C}, \text{ almost surely}.
    \end{equation}
\end{theorem}

Allocating infinitely many samples to each alternative-context pair would have been sufficient to prove consistency if we did not model any correlation between contexts. However, we are not aware of any general consistency results regarding the GP posterior mean, so we need to ensure that the predictions converge to the true reward function. The following lemma establishes this result.

\begin{lemma}\label{lemma-posterior-mean-consistency}
    Under Assumption \ref{assumption-consistency}, if $\lim_{n \rightarrow \infty} N_n(k, c) = \infty$, then $\mu_{n}(k, c) \rightarrow F(k, c)$ almost surely.
\end{lemma}

Coupled with Theorem \ref{theorem-infinitely-sample}, Lemma \ref{lemma-posterior-mean-consistency} shows that using GP-C-OCBA to allocate the observation budget, the posterior mean converges almost surely to the true reward for all alternative-context pairs. Thus, it follows that the predicted best alternative must  almost surely converge to the correct best alternative for all contexts. The following corollary summarizes the result.

\begin{corollary}
    Under conditions of Theorem \ref{theorem-infinitely-sample}, GP-C-OCBA policy is strongly consistent, i.e., $\pi^B(c) \rightarrow pi^*(c), \forall \ c$ almost surely as the observation budget $B \rightarrow \infty$.
\end{corollary}

\subsection{Convergence to the Optimal Allocation Ratio and the Convergence Rate}
The results of the previous subsection show that, given enough samples, GP-C-OCBA is guaranteed to identify the optimal policy, $\pi^*(c)$. However, these are asymptotic results, and they do not provide any insight into how fast this convergence occurs. Under certain conditions on the prior covariance matrix, we can show that the allocation ratio produced by GP-C-OCBA converges to the optimal static allocation ratio, which is obtained by maximizing the rate function given in Theorem \ref{thm-pcs-rate-function}.

Let
\begin{equation*}
    \eta(k, c) = \frac{(F(\pi^*(c), c) - F(k, c))^2}{\sigma^2(\pi^*(c), c) / p^*(\pi^*(c), c) + \sigma^2(k, c) / p^*(k, c)}
\end{equation*}
and
\begin{equation*}
    \hat{\eta}_n(k, c) = \frac{(F(\pi^*(c), c) - F(k, c))^2}{\sigma^2(\pi^*(c), c) / \hat{p}_n(\pi^*(c), c) + \sigma^2(k, c) / \hat{p}_n(k, c)}.
\end{equation*}
The work of \cite{Gao2019Covariates} shows that the optimal static allocation ratio, i.e., the allocation ratio maximizing the large deviations rate function, satisfies the following conditions.
\begin{lemma}\label{lemma-optimal-allocation} (Lemma 1 of \cite{Gao2019Covariates})
    The optimal static allocation ratio, $p^*(k, c)$, satisfies the following.
    \begin{align*}
        \frac{p^*(\pi^*(c), c)^2}{\sigma^2(\pi^*(c), c)} &= \sum_{k \neq \pi^*(c)} \frac{p^*(k, c)^2}{\sigma^2(k, c)}, \forall c \in \mathcal{C} \\
        \eta(k, c) &= \eta(k', c), \forall c \in \mathcal{C}, k \neq k' \neq \pi^*(c) \\
        \eta(k, c) &= \eta(k', c'), \forall c \neq c', k \neq \pi^*(c), k' \neq \pi^*(c').
    \end{align*}
\end{lemma}

The following result states that if the GP prior assumes no correlation between contexts, then the allocation ratio produced by GP-C-OCBA converges to $p^*(k, c)$, i.e., that GP-C-OCBA asymptotically obtains the optimal static allocation ratio.

\begin{theorem}\label{theorem-optimal-rate}
    Suppose that Assumption \ref{assumption-consistency} holds, and the GP prior satisfies \linebreak $\Sigma_0(\mathcal{C}, \mathcal{C}; k) = diag([\Sigma_0(k, c)]_{c \in \mathcal{C}}), \forall k$. Then, 
    \begin{align*}
        \frac{\hat{p}_n(\pi^*(c), c)^2}{\sigma^2(\pi^*(c), c)} - \sum_{k \neq \pi^*(c)} \frac{\hat{p}_n(k, c)^2}{\sigma^2(k, c)} &\rightarrow 0, \forall c \in \mathcal{C} \\
        \hat{\eta}_n(k, c) - \hat{\eta}_n(k', c) &\rightarrow 0, \forall c \in \mathcal{C}, k \neq k' \neq \pi^*(c) \\
        \hat{\eta}_n(k, c) - \hat{\eta}_n(k', c') &\rightarrow 0, \forall c \neq c', k \neq \pi^*(c), k' \neq \pi^*(c'),
    \end{align*}
    almost surely as $n \rightarrow \infty$. In other words, $\hat{p}_n(k, c)$ satisfies the conditions of Lemma \ref{lemma-optimal-allocation} and converges to $p^*(k, c)$ as $n \rightarrow \infty$.
\end{theorem}
\begin{proof}
    If we use an uninformative prior, i.e., $\Sigma_0(k, c) = \infty$, then $\mu_n(\cdot, \cdot)$ and $\Sigma_n(\cdot, \cdot)$ simplify to the sample mean, $\widebar{Y}_n(k, c)$, and $\sigma^2(k, c) / N_n(k, c)$, respectively. In which case, GP-C-OCBA is identical to C-OCBA with known observation noise, and the result follows from Theorem 2 of \cite{Gao2019Covariates}. In the case when $\Sigma_0(k, c) < \infty$, the estimators involve additional terms for the prior, i.e., $\mu_n(k, c) = \frac{\mu_0(k, c)/\Sigma_0(k, c) + N_n(k, c) \widebar{Y}_n(k, c) / \sigma^2(k, c)}{1 / \Sigma_0(k, c) + N_n(k, c) / \sigma^2(k, c)}$ and $\Sigma_n(k, c) = \frac{1}{1 / \Sigma_0(k, c) + N_n(k, c) / \sigma^2(k, c)}$. These are asymptotically equivalent to $\mu_n(k, c) \xrightarrow{\approx} \widebar{Y}_n(k, c)$ and $\Sigma_n(k, c) \xrightarrow{\approx} \sigma^2(k, c) / N_n(k, c)$. Thus, the effects of the prior terms disappear in the asymptotic analysis and the same proof applies.
\end{proof}

Although Theorem \ref{theorem-optimal-rate} shows that the sample allocation ratio of GP-C-OCBA converges to the optimal allocation ratio, this, by itself, does not say much about the actual convergence rate of the contextual $PCS$. In the following, we extend the convergence rate result presented by \cite{li-gao2021convergence} for the OCBA algorithm to the contextual setting. We show that the contextual $PFS$ converges to $0$ at the optimal exponential rate. To the best of our knowledge, this is the first convergence rate result in the literature for a contextual R\&S algorithm.

\begin{theorem} \label{theorem-convergence-rate}
    Under the conditions of Theorem \ref{theorem-optimal-rate}, and using an uninformative prior, i.e., $\Sigma_0(k, c) = \infty$, the contextual probability of false selection produced by GP-C-OCBA policy decreases at the optimal exponential rate, i.e.,
    \begin{equation*}
        PFS_{\sim}^n \doteq e^{- \eta^* n / 2}, \text{ where } \eta^* = \min_{c \in \mathcal{C}} \min_{k \neq \pi^*(c)} \eta(k, c),
    \end{equation*}
    where $PFS^n_{\sim}$ is either of $PFS^n_E$ or $PFS^n_M$, and $\doteq$ denotes logarithmic equivalence, i.e., for two sequences $a_n \doteq b_n \iff \lim_{n \rightarrow \infty} \frac{1}{n} \log \frac{a_n}{b_n} = 0$. 
\end{theorem}
\begin{proof}
    Recall from the proof of Theorem \ref{thm-pcs-rate-function} that $PFS^n_{\sim}$ can be lower and upper bounded by a constant ($1$ and $|\mathcal{C}|(|\mathcal{K}| - 1)$, respectively) multiple of
    \begin{equation*}
        \max_{c \in \mathcal{C}}\max_{k \neq \pi^*(c)} P(\mu_n(\pi^*(c), c) < \mu_n(k, c)).
    \end{equation*}
    With the prior as specified, the posterior mean is equal to the sample mean. In which case \cite{Chen2010OCBA}
    \begin{equation*}
        P(\mu_n(\pi^*(c), c) < \mu_n(k, c)) = \Phi \left( - \sqrt{\hat{\eta}_n(k, c) n} \right),
    \end{equation*}
    where $\Phi(\cdot)$ denotes the cumulative distribution function of the standard Gaussian distribution. 
    Since for $x < 0$, $\Phi(x) \doteq \exp(- x^2 / 2)$ \cite{Gordon1941}, we get
    \begin{equation*}
        P(\mu_n(\pi^*(c), c) < \mu_n(k, c)) \doteq \exp \left( - \hat{\eta}_n(k, c) n / 2 \right).
    \end{equation*}
    By Theorem \ref{theorem-optimal-rate}, $\hat{p}_n(k, c) \rightarrow p^*(k, c)$, which implies that $\hat{\eta}_n(k, c) \rightarrow \eta(k, c)$ as $n \rightarrow \infty$. Thus, putting it all together, we get 
    \begin{equation*}
        PFS^n_{\sim} \doteq \exp \left(- \left(\min_{c \in \mathcal{C}} \min_{k \neq \pi^*(c)} \eta(k, c)\right)n / 2 \right).
    \end{equation*}
\end{proof}

Note that the Theorem \ref{theorem-optimal-rate} is an extension of the result originally proved for C-OCBA \cite{Gao2019Covariates} and the convergence rate result in Theorem \ref{theorem-convergence-rate} also applies to C-OCBA under the assumption of known observation noise.
This can be interpreted as saying that GP-C-OCBA and C-OCBA are asymptotically equivalent. A natural question in this case is whether GP-C-OCBA can be shown to have an advantage over C-OCBA under a finite budget setting.
Though it is difficult to show this rigorously, we can argue for it in an informal way. By modeling the performance of an alternative-context with a Gaussian process, we implicitly assume that the contexts that are similar will have similar performances, with the particular measure of similarity being learned from the data. As a result, for a given alternative-context, by learning about the contexts that are similar to this context, we can refine our model of the contexts without having to evaluate the context itself.
This is similar to the linearity assumptions that we see in the literature, e.g., in \cite{shen2021ranking}. In their case, there's an explicit assumption of linearity, which, if holds true, allows for evaluating only the extreme contexts to determine the policy for all contexts. Due to the assumed linearity, the performance of any given context can be estimated by solving the regression equations. There are many cases where linearity provides a good approximation, but there are many others where it does not hold true at all. Our modeling of similarity, as opposed to linearity, works in a similar way. It does not eliminate the need to evaluate all contexts. However, by allowing inference on one context's performance via other similar contexts' evaluations, it reduces the number of evaluations needed, leading to a better finite time performance.
An example of this was shown in Figure \ref{fig:gp_example}, where we saw that the GP posterior mean produced predictions that resembled the true rewards much more closely, while the sample mean predictor required many more samples to refine the estimates.

\section{NUMERICAL EXPERIMENTS} \label{section-numerical-experiments}

In this section, we demonstrate the performance of our algorithm on a set of synthetic benchmark problems. We compare our algorithm with the algorithms by \cite{li2020contextdependent} (DSCO), \cite{Gao2019Covariates} (C-OCBA), and with the IKG algorithm as described in Section \ref{section-ikg}. In addition, we also compare against a modified version of the two-stage algorithms (TS and TS+) proposed by \cite{shen2021ranking}. Both TS and TS+ are designed for the fixed confidence setting under a linearity assumption on the reward function. We adopt them by keeping the allocation ratios suggested by the algorithms but scaling down the number of samples to the given budget. As recommended by the authors, we use the extreme design for both TS and TS+, with the extreme design being restricted to the points in the context set.
We chose these benchmarks since DSCO and C-OCBA were both proposed for the contextual R\&S with the finite alternative-context setting that is studied in this paper and has demonstrated superior performance in experiments; and IKG was chosen since KG type algorithms, including variants of IKG, have consistently demonstrated superior sampling efficiency under various problem settings. 

We implemented the experiments in Python, and used the GP models from the BoTorch package \cite{balandat2020botorch} with the default priors. The GP hyper-parameters are re-trained every $10$ iterations, and we use the Matern $5/2$ kernel. The code will be made available at \url{https://github.com/saitcakmak/contextual_rs} upon publication.

\subsection{Synthetic Test Functions}
For the experiments, we generate the true rewards, $F(k, c)$, by evaluating common global optimization test functions on randomly drawn points from the function domain.
We use the first dimension of the function input for the alternatives, i.e. each alternative corresponds to a fixed value of $x_1$, and spread the alternatives evenly across the corresponding domain. The remaining input dimensions are used for the contexts, and thus contexts are $d-1$ dimensional vectors for a $d$ dimensional test function. Put together, this corresponds to $F(k, c) = f(x_k, x_c)$ where $x_k$ and $x_c$ are fixed realizations of $1$ and $d-1$ dimensional uniform random variables, respectively. 
The rewards are observed with additive Gaussian noise with standard deviation set as $\frac{f_{max} - f_{min}}{100 / 3}$,
where $f_{max}$ and $f_{min}$ are estimated using $1000$ samples drawn uniformly at random from the function domain. We use the following functions in our experiments:

\begin{itemize} 
    \item The 2D Branin function, evaluated on $[-5, 10] \times [0, 10]$:
    \begin{equation*}
        f(x) = - (x_2 - b x_1^2 + c x_1 - r)^2 - 10 (1-t) cos(x_1) - 10,
    \end{equation*}
    where $b = 5.1 / (4 \pi^2)$, $c = 5 / \pi$, $r = 6$ and $t = 1 / (8 \pi)$. We run two experiments using the Branin function, both with 10 alternatives and 10 contexts. The first objective is the expected $PCS$ with weights set arbitrarily as $[0.03, 0.07, 0.2, 0.1, 0.15, 0.2, 0.02, 0.08, 0.1, 0.05]$, and the second objective is the worst-case $PCS$. We draw $2$ samples from each alternative-context pair for the initialization phase. For TS and TS+, the extreme design consists of the smallest and largest context values, with each receiving $10$ samples from each alternative for the initialization phase.
    
    \item The 2D Griewank function, evaluated on $[-10, 10]^2$:
    \begin{equation*}
        f(x) = - \sum_{i=1}^d \frac{x_i^2}{4000} + \prod_{i=1}^d \cos \left( \frac{x_i}{\sqrt{i}} \right) - 1.
    \end{equation*}
    We run two experiments with the Greiwank function, using 10 alternatives and 20 contexts. We use the expected $PCS$ with uniform weights and the worst-case $PCS$, and initialize with $2$ samples from each alternative-context pair. For TS and TS+, the extreme design consists of the smallest and largest context values, with each receiving $20$ samples from each alternative for the initialization phase.
    
    \item The 3D Hartmann function, evaluated on $[0, 1]^3$:
    \begin{equation*}
        f(x) = \sum_{i=1}^4 \alpha_i \exp \left(-\sum_{j=1}^3 A_{ij} (x_j - P_ij)^2 \right),
    \end{equation*}
    where the constants $\alpha$, $A$, and $P$ are given in \cite{test_functions}.
    We run a single experiment with 20 alternatives and 20 contexts, using the expected $PCS$ with uniform weights. Since the number of alternative-context pairs is quite large in this experiment, we select only $6$ contexts for each alternative, uniformly at random, and draw a single sample from these contexts for the initial stage. Due to insufficient initial sampling budget, we do not run DSCO and C-OCBA for this problem. Since it is difficult to define the extreme design over a discrete set of points in $2$ dimensions, we skip TS and TS+ as well and only run the GP based algorithms for this problem.
    
    \item The 8D Cosine8 function, evaluated on $[-1, 1]^8$:
    \begin{equation*}
        f(x) = 0.1 \sum_{i=1}^8 cos(5 \pi x_i) - \sum_{i=1}^8 x_i^2.
    \end{equation*}
    We run a single experiment with the mean $PCS$ objective with uniform weights. We use 20 alternatives and 40 contexts. For the initial stage, we randomly select $16$ contexts for each alternative, and draw a single sample from these contexts, which is again due to the large number of alternative-context pairs in this experiment. Similar to the previous experiment, due to insufficient initial sampling budget, we do not run DSCO and C-OCBA for this problem. Since it is difficult to define the extreme design over a discrete set of points in $7$ dimensions, we skip TS and TS+ as well and only run the GP based algorithms for this problem.
\end{itemize}

\begin{figure}[htb]
    \centering
      \includegraphics[width=\textwidth]{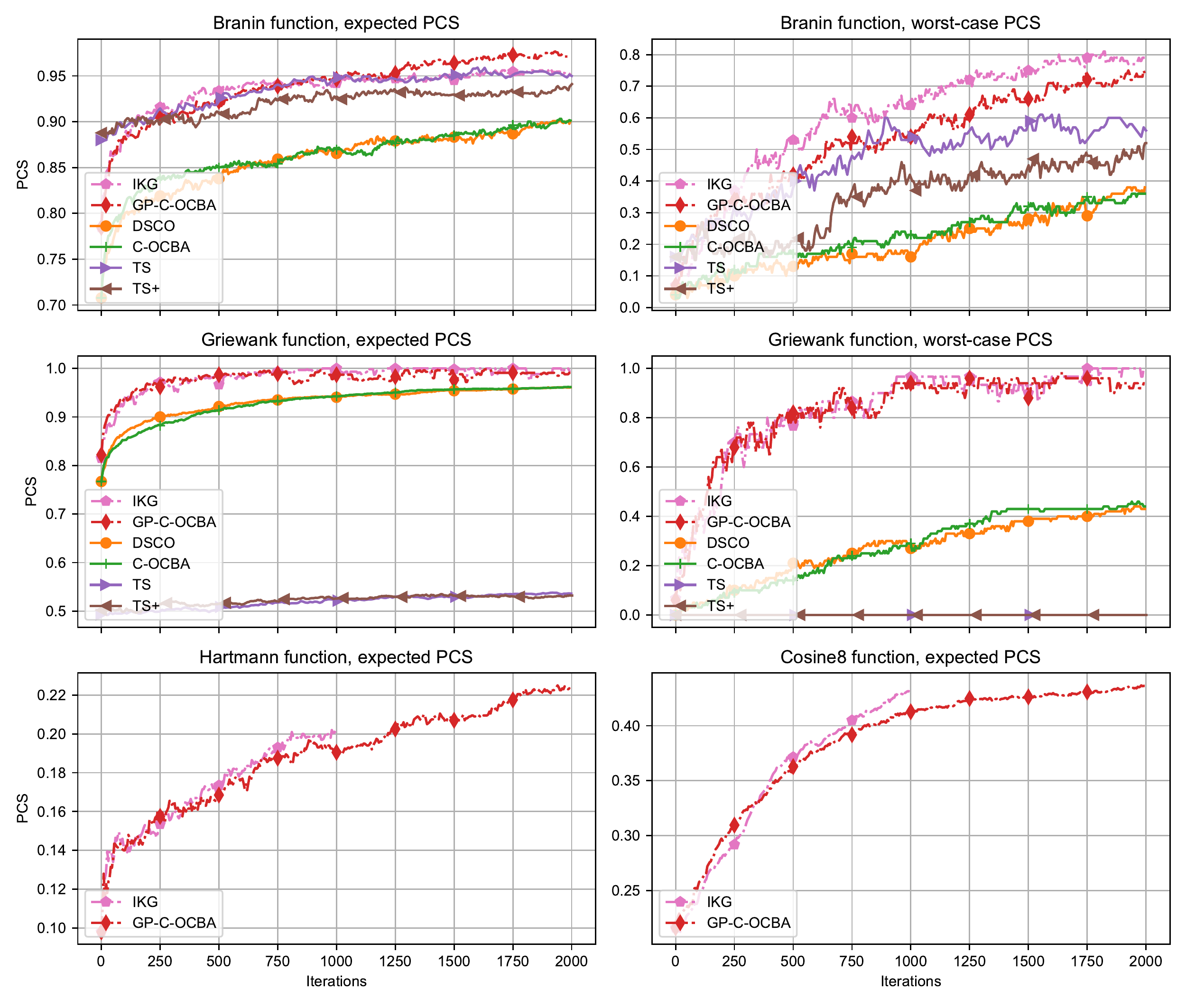}
    \caption{
    Experiments using Branin function with $PCS_E$ and $PCS_M$, Griewank function with $PCS_E$ and $PCS_M$, Hartmann function with $PCS_E$, and Cosine8 with $PCS_E$. 
    The plots show the empirical contextual $PCS$ on the y-axis, and the number of iterations/samples (post-initialization) on the x-axis.}
    \label{fig:exp-results}
\end{figure}

\subsection{Results of Synthetic Test Functions}
The experiment results are plotted in Figure \ref{fig:exp-results}. We ran each experiment for $2000$ iterations, except for IKG in Hartmann and Cosine8 functions, which were run for $1000$ iterations due to their excessive cost. The plots show the empirical contextual $PCS$, estimated using 100 replications. The first $4$ plots compare all algorithms. However, the last $2$ only compare GP-C-OCBA and IKG, due to small initial budget preventing drawing of multiple samples from each alternative-context pair, which is necessary to form an initial estimate of sample mean and variance that is used by DSCO and C-OCBA, as well as the difficulty of setting up the extreme design for the two-stage algorithms. In all $4$ of the experiments comparing all algorithms, we see that the two algorithms using the GP models achieve the highest contextual $PCS$. This demonstrates the benefit of using a statistical model that leverages the hidden correlation structure in the reward function. In particular, for the worst-case $PCS$, we see that the DSCO and C-OCBA perform significantly worse in both problems, with the performance of TS and TS+ also diminishing significantly. For TS and TS+, this clearly demonstrates the weakness of using a linear model when the underlying function is non-linear, which leads to $0$ worst-case $PCS$ in the Griewank function. Similarly, the poor performance of DSCO and C-OCBA in the worst-case $PCS$ is explained by a total of $2400$ samples being far from sufficient to form reliable estimates for $200$ alternative-context pairs using an independent statistical model. 

Although it is slightly trailing behind in some experiments, we see that GP-C-OCBA is highly competitive against IKG, while having significantly smaller computational complexity. The wall-clock times for the experiments are reported in Table \ref{table:wall-clock-time}. We see that even in the experiments with smallest number of alternatives and contexts, the IKG algorithm takes about $5$ times as long to run, with the ratio increasing significantly to about $75$ times as we move to larger experiments. The reported run times are for $1000$ iterations of a full experiment, and include the cost of fitting the GP model, which is identical for both algorithms. It is worth noting that cost of evaluating the test functions is in these experiments is insignificant compared to the running time. As the cost of function evaluation increases, the results would look nicer for IKG, though it would still remain the more expensive alternative.

\begin{table}
	\caption{Comparison of computational cost of IKG and GP-C-OCBA. We report the average wall-clock time, in seconds, for running $1000$ iterations of the given experiment. The experiments were run on a shared cluster using $4$ cores of the allocated CPU. To save on space, we report the average run-time (in seconds) of the two objectives for Branin and Griewank.}
	\label{table:wall-clock-time} \centering
\begin{tabular}{ccccc}
    \toprule
	Algorithm & Branin $PCS_E$ / $PCS_M$ & Griewank $PCS_E$ / $PCS_M$ & Hartmann & Cosine8 \\
	\midrule
	IKG & 1218 & 3807 & 11096 & 43224 \\
	GP-C-OCBA & 249 & 293 & 441 & 544 \\
    \bottomrule
\end{tabular}
\end{table}

Overall, the experiments show that GP-C-OCBA is highly competitive in terms of sampling efficiency, while being significantly cheaper than other high performing benchmarks. We believe that this makes GP-C-OCBA an attractive option for any practitioner that is faced with a contextual R\&S problem.

\subsection{Covid-19 Testing Allocation}
We consider the problem of allocating a fixed supply of Covid-19 tests between three populations, with the aim of minimizing the total number of infections by the end of the testing period of two weeks. 
We have access to a simulator that can be used to simulate the disease spread within the populations and to predict the number of infections. 
The simulator takes in a set of disease and population parameters (assumed fixed), the probability of an exposed person getting infected, the expected number of daily contacts for each person (the contexts), and the allocation of the testing capacity between populations (the alternatives / designs); simulates the interactions within each population, the progression of the infected individuals through disease stages, quarantining of the individuals that tests positive etc.; and reports the total number of infected individuals. 

The simulator was originally developed by \cite{covid-github}, and the problem setting is a modification from \cite{Cakmak2020borisk}, which studied it under a risk-averse setting. We consider the case where we have a finite simulation budget to experiment with possible testing allocations under any of the $16$ possible contexts, which are combinations of the exposed infection probability and the expected daily contacts, which independently each take $4$ equally spaced values range from $[0.015, 0.03]$ and $[6, 12]$, respectively. 
The effective context will be determined based on an analysis of recent positive cases, the prevalence of different strains of the virus leading to different levels of infectiousness, as well as the recent contact tracing and survey data, that will be used to estimate the number of daily contacts. 
Once the true context / disease prevalence is revealed, the tests must be allocated immediately, using the best contextual allocation obtained during the experimentation phase. It is assumed that there are a total of $10000$ tests, which must be allocated in lots of $1000$'s. For fairness reasons, each population must be allocated at least $20\%$ and at most $50\%$ of the tests, leading to a total of $12$ possible designs. 

To initialize the algorithms, we use $2$ samples from each alternative-context pair, except for TS and TS+ where the contexts are restricted to the $4$ extreme designs, which are the corner points of the context grid. We plot the expected $PCS$ obtained by the algorithms against both the total number of samples used (including initialization), as well as against the wall-clock time w.r.t. varying values of the simulation time.

\begin{figure}[htb]
    \centering
      \includegraphics[width=0.6\textwidth]{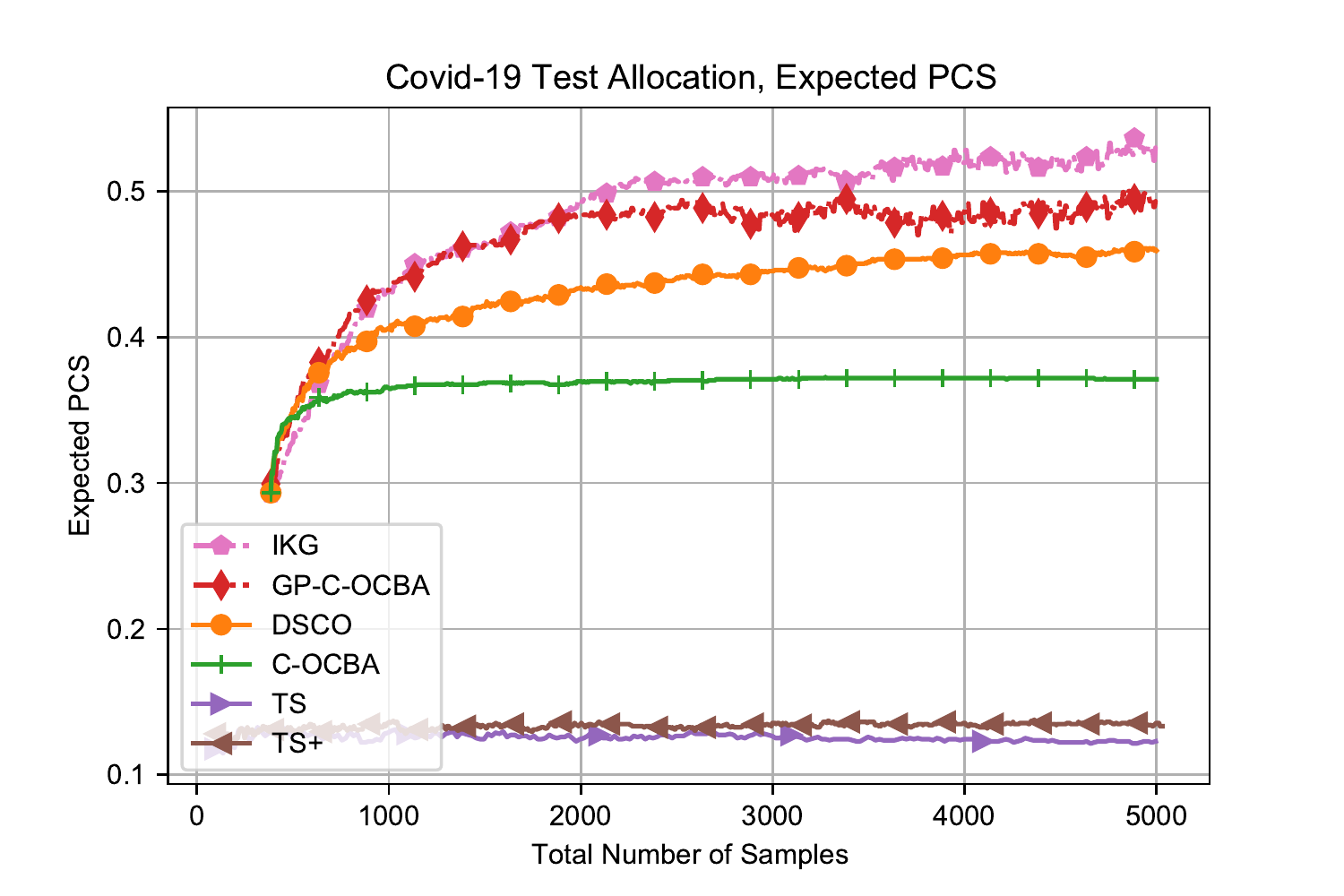}
    \caption{The expected $PCS$ in the Covid-19 Testing Allocation problem, plotted against the total number of samples used.}
    \label{fig:covid_iters}
\end{figure}

\begin{figure}[htb]
    \centering
      \includegraphics[width=\textwidth]{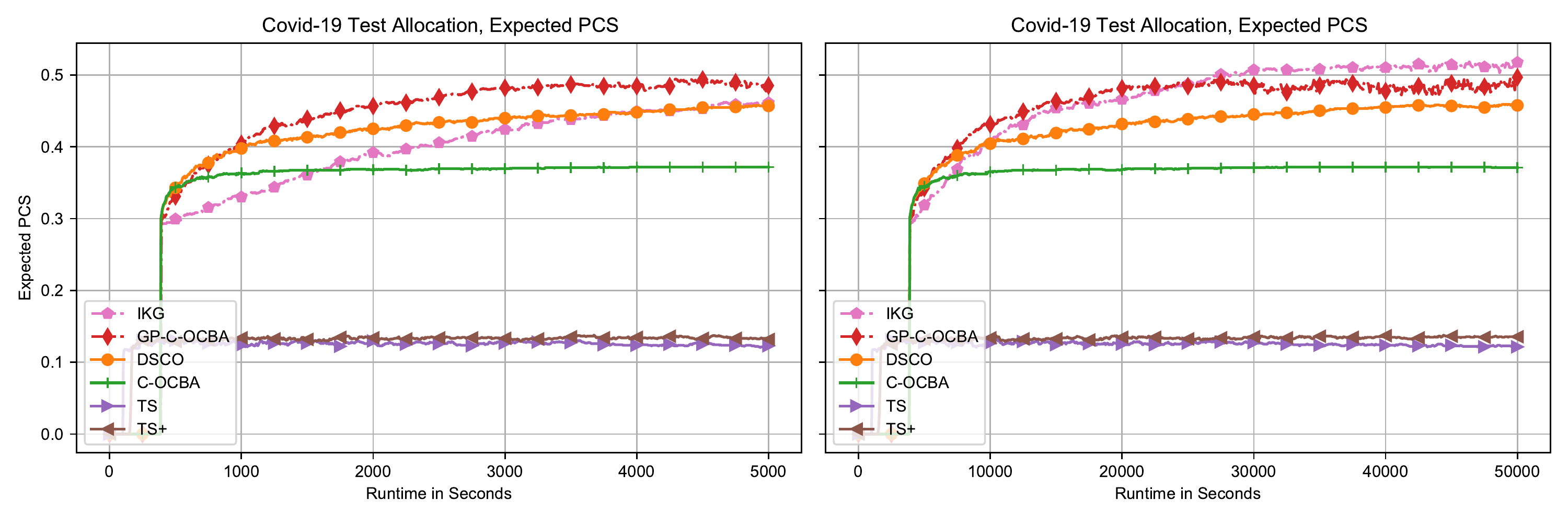}
    \caption{The expected $PCS$ in the Covid-19 Testing Allocation problem, plotted against the total wall-time (in seconds). The simulation time is $1$ and $10$ seconds for the left and right plots, respectively.}
    \label{fig:covid_time}
\end{figure}

The results, plotted in Figures \ref{fig:covid_iters} \& \ref{fig:covid_time} show that GP-C-OCBA overall remains very competitive. The IKG performs the best in terms of the total number of samples used, though its performance is overshadowed by its run-time unless the simulations are expensive enough. For the rest of the algorithms, with the exception of DSCO, we do not see a significant improvement in performance as more samples are allocated. 
This can be explained by the statistical models used. Compared to GP-C-OCBA, C-OCBA requires many samples to refine the performance estimates due to total independence in the statistical model. 
For TS and TS+, the issue can be attributed to the underlying linearity assumption, which does not hold in this case. Judging by the $PCS$, the linear model fails to accurately estimate the performances even in the extreme designs that were used to sample, which, if estimated perfectly, would lead to an expected $PCS$ of $0.25$. In light of these results, we caution on using a linear model unless there is strong evidence that the underlying reward function satisfies the linear structure assumed by the model (e.g., is a linear function of contexts), and instead recommend using algorithms based on more flexible non-linear GP models, such as the GP-C-OCBA and IKG algorithms.

\section{CONCLUSION} \label{section-conclusion}
We studied the contextual R\&S problem under finite alternative-context setting, using a separate GP to model the reward for each alternative. We derived the large deviations rate functions for the contextual $PCS$, and proposed the GP-C-OCBA algorithm that aims to maximize the rate function using the information available in the GP posterior. 
We proved consistency of GP-C-OCBA. We showed that its allocations converge to the optimal allocation ratio and the resulting contextual PFS converges to $0$ at the optimal exponential rate. 
In numerical experiments using both synthetic test functions and a Covid-19 simulator, GP-C-OCBA was shown to achieve significant sampling efficiency, while having a significantly smaller computational overhead compared to other competitive alternatives.

\begin{acks}
 The first and third authors gratefully acknowledge the support by \grantsponsor{afosr}{Air Force Office of Scientific Research}{https://www.afrl.af.mil/AFOSR/} under Grant \grantnum[https://www.afrl.af.mil/AFOSR/]{afosr}{FA9550-19-1-0283} and the support by \grantsponsor{nsf}{National Science Foundation}{https://www.nsf.gov} under Grant \grantnum[https://www.nsf.gov]{nsf}{DMS2053489}. 
\end{acks}

\bibliographystyle{ACM-Reference-Format}
\bibliography{ref}  


\begin{thebibliography}{24}


\ifx \showCODEN    \undefined \def \showCODEN     #1{\unskip}     \fi
\ifx \showDOI      \undefined \def \showDOI       #1{#1}\fi
\ifx \showISBNx    \undefined \def \showISBNx     #1{\unskip}     \fi
\ifx \showISBNxiii \undefined \def \showISBNxiii  #1{\unskip}     \fi
\ifx \showISSN     \undefined \def \showISSN      #1{\unskip}     \fi
\ifx \showLCCN     \undefined \def \showLCCN      #1{\unskip}     \fi
\ifx \shownote     \undefined \def \shownote      #1{#1}          \fi
\ifx \showarticletitle \undefined \def \showarticletitle #1{#1}   \fi
\ifx \showURL      \undefined \def \showURL       {\relax}        \fi
\providecommand\bibfield[2]{#2}
\providecommand\bibinfo[2]{#2}
\providecommand\natexlab[1]{#1}
\providecommand\showeprint[2][]{arXiv:#2}

\bibitem[Balandat et~al\mbox{.}(2020)]%
        {balandat2020botorch}
\bibfield{author}{\bibinfo{person}{Maximilian Balandat}, \bibinfo{person}{Brian
  Karrer}, \bibinfo{person}{Daniel~R. Jiang}, \bibinfo{person}{Samuel Daulton},
  \bibinfo{person}{Benjamin Letham}, \bibinfo{person}{Andrew~Gordon Wilson},
  {and} \bibinfo{person}{Eytan Bakshy}.} \bibinfo{year}{2020}\natexlab{}.
\newblock \showarticletitle{{BoTorch: A Framework for Efficient Monte-Carlo
  Bayesian Optimization}}. In \bibinfo{booktitle}{\emph{Advances in Neural
  Information Processing Systems 33}}. \bibinfo{publisher}{Curran Associates,
  Inc.}, \bibinfo{address}{Red Hook, NY}, \bibinfo{pages}{21524--21538}.
\newblock


\bibitem[Bect et~al\mbox{.}(2019)]%
        {bect2019GPSupermartingale}
\bibfield{author}{\bibinfo{person}{Julien Bect}, \bibinfo{person}{François
  Bachoc}, {and} \bibinfo{person}{David Ginsbourger}.}
  \bibinfo{year}{2019}\natexlab{}.
\newblock \showarticletitle{A Supermartingale Approach to Gaussian Process
  Based Sequential Design of Experiments}.
\newblock \bibinfo{journal}{\emph{Bernoulli}}  \bibinfo{volume}{25}
  (\bibinfo{year}{2019}), \bibinfo{pages}{2883--2919}.
\newblock


\bibitem[Cakmak et~al\mbox{.}(2020)]%
        {Cakmak2020borisk}
\bibfield{author}{\bibinfo{person}{Sait Cakmak}, \bibinfo{person}{Raul
  Astudillo}, \bibinfo{person}{Peter Frazier}, {and} \bibinfo{person}{Enlu
  Zhou}.} \bibinfo{year}{2020}\natexlab{}.
\newblock \showarticletitle{Bayesian Optimization of Risk Measures}. In
  \bibinfo{booktitle}{\emph{Advances in {N}eural {I}nformation {P}rocessing
  {S}ystems 33}} (Vancouver, Canada). \bibinfo{publisher}{Curran Associates,
  Inc.}, \bibinfo{address}{Red Hook, NY}, \bibinfo{pages}{20130--20141}.
\newblock


\bibitem[Cashore et~al\mbox{.}(2020)]%
        {covid-github}
\bibfield{author}{\bibinfo{person}{Massey Cashore}, \bibinfo{person}{Peter
  Frazier}, \bibinfo{person}{Yujia Zhang}, {and} \bibinfo{person}{Jiayue Wan}.}
  \bibinfo{year}{2020}\natexlab{}.
\newblock \bibinfo{title}{Group Testing}.
\newblock
\newblock
\newblock
\shownote{\textit{Code available at}
  \url{https://github.com/peter-i-frazier/group-testing/}, \textit{white-paper
  available at}
  \url{https://people.orie.cornell.edu/pfrazier/COVID_19_Modeling_Jun15.pdf}}.


\bibitem[Chen et~al\mbox{.}(2015)]%
        {Chen2015RSoverview}
\bibfield{author}{\bibinfo{person}{Chun-Hung Chen}, \bibinfo{person}{Stephen~E.
  Chick}, \bibinfo{person}{Loo~Hay Lee}, {and} \bibinfo{person}{Nugroho~A.
  Pujowidianto}.} \bibinfo{year}{2015}\natexlab{}.
\newblock \showarticletitle{Ranking and Selection: Efficient Simulation Budget
  Allocation}. In \bibinfo{booktitle}{\emph{Handbook of Simulation
  Optimization}}, \bibfield{editor}{\bibinfo{person}{Michael~C Fu}} (Ed.).
  \bibinfo{publisher}{Springer}, \bibinfo{address}{New York, NY},
  \bibinfo{pages}{45--80}.
\newblock


\bibitem[Chen and Lee(2010)]%
        {Chen2010OCBA}
\bibfield{author}{\bibinfo{person}{Chun-Hung Chen} {and}
  \bibinfo{person}{Loo~Hay Lee}.} \bibinfo{year}{2010}\natexlab{}.
\newblock \bibinfo{booktitle}{\emph{Stochastic Simulation Optimization}}.
\newblock \bibinfo{publisher}{World Scientific}.
\newblock
\urldef\tempurl%
\url{https://doi.org/10.1142/7437}
\showDOI{\tempurl}


\bibitem[Dembo and Zeitouni(1998)]%
        {dembo1998Largedev}
\bibfield{author}{\bibinfo{person}{Amir Dembo} {and} \bibinfo{person}{Ofer
  Zeitouni}.} \bibinfo{year}{1998}\natexlab{}.
\newblock \bibinfo{booktitle}{\emph{Large Deviations Techniques and
  Applications} (\bibinfo{edition}{2nd} ed.)}.
\newblock \bibinfo{publisher}{Springer-Verlag}, \bibinfo{address}{Berlin}.
\newblock


\bibitem[Ding et~al\mbox{.}(2021)]%
        {zhang2020IKG}
\bibfield{author}{\bibinfo{person}{Liang Ding}, \bibinfo{person}{L.~Jeff Hong},
  \bibinfo{person}{Haihui Shen}, {and} \bibinfo{person}{Xiaowei Zhang}.}
  \bibinfo{year}{2021}\natexlab{}.
\newblock \showarticletitle{Technical Note -- Knowledge Gradient for Selection
  with Covariates: Consistency and Computation}.
\newblock \bibinfo{journal}{\emph{Naval Research Logistics}}
  (\bibinfo{year}{2021}).
\newblock


\bibitem[Feng et~al\mbox{.}(2020)]%
        {Feng2020ContextualBO}
\bibfield{author}{\bibinfo{person}{Qing Feng}, \bibinfo{person}{Ben Letham},
  \bibinfo{person}{Hongzi Mao}, {and} \bibinfo{person}{Eytan Bakshy}.}
  \bibinfo{year}{2020}\natexlab{}.
\newblock \showarticletitle{High-Dimensional Contextual Policy Search with
  Unknown Context Rewards using Bayesian Optimization}. In
  \bibinfo{booktitle}{\emph{Advances in Neural Information Processing Systems
  33}}. \bibinfo{publisher}{Curran Associates, Inc.}, \bibinfo{address}{Red
  Hook, NY}, \bibinfo{pages}{22032--22044}.
\newblock


\bibitem[Frazier et~al\mbox{.}(2009)]%
        {Frazier2009KG}
\bibfield{author}{\bibinfo{person}{Peter Frazier}, \bibinfo{person}{Warren
  Powell}, {and} \bibinfo{person}{Savas Dayanik}.}
  \bibinfo{year}{2009}\natexlab{}.
\newblock \showarticletitle{The Knowledge-Gradient Policy for Correlated Normal
  Beliefs}.
\newblock \bibinfo{journal}{\emph{INFORMS Journal on Computing}}
  \bibinfo{volume}{21}, \bibinfo{number}{4} (\bibinfo{year}{2009}),
  \bibinfo{pages}{599--613}.
\newblock


\bibitem[Gao et~al\mbox{.}(2019)]%
        {Gao2019Covariates}
\bibfield{author}{\bibinfo{person}{Siyang Gao}, \bibinfo{person}{Jianzhong Du},
  {and} \bibinfo{person}{Chun-Hung Chen}.} \bibinfo{year}{2019}\natexlab{}.
\newblock \showarticletitle{Selecting the Optimal System Design under
  Covariates}. In \bibinfo{booktitle}{\emph{15\textsuperscript{th}
  International Conference on Automation Science and Engineering}}.
\newblock


\bibitem[Glynn and Juneja(2004)]%
        {Glynn2004LargeDev}
\bibfield{author}{\bibinfo{person}{Peter Glynn} {and} \bibinfo{person}{Sandeep
  Juneja}.} \bibinfo{year}{2004}\natexlab{}.
\newblock \showarticletitle{A Large Deviations Perspective on Ordinal
  Optimization}. In \bibinfo{booktitle}{\emph{Proceedings of the 2004 Winter
  Simulation Conference}}. \bibinfo{publisher}{Institute of Electrical and
  Electronics Engineers, Inc}, \bibinfo{address}{Piscataway, NJ},
  \bibinfo{pages}{577–585}.
\newblock


\bibitem[Gordon(1941)]%
        {Gordon1941}
\bibfield{author}{\bibinfo{person}{Robert~D. Gordon}.}
  \bibinfo{year}{1941}\natexlab{}.
\newblock \showarticletitle{{Values of Mills' Ratio of Area to Bounding
  Ordinate and of the Normal Probability Integral for Large Values of the
  Argument}}.
\newblock \bibinfo{journal}{\emph{The Annals of Mathematical Statistics}}
  \bibinfo{volume}{12}, \bibinfo{number}{3} (\bibinfo{year}{1941}),
  \bibinfo{pages}{364 -- 366}.
\newblock


\bibitem[Guo and Berkhahn(2016)]%
        {Guo2016Embedding}
\bibfield{author}{\bibinfo{person}{Cheng Guo} {and} \bibinfo{person}{Felix
  Berkhahn}.} \bibinfo{year}{2016}\natexlab{}.
\newblock \showarticletitle{Entity Embeddings of Categorical Variables}.
\newblock  (\bibinfo{year}{2016}).
\newblock
\showeprint[arxiv]{1604.06737}


\bibitem[Jin et~al\mbox{.}(2019)]%
        {Jin2019Analytics}
\bibfield{author}{\bibinfo{person}{Xiao Jin}, \bibinfo{person}{Haobin Li},
  {and} \bibinfo{person}{Loo~Hay Lee}.} \bibinfo{year}{2019}\natexlab{}.
\newblock \showarticletitle{Optimal Budget Allocation in Simulation
  Analytics*}. In \bibinfo{booktitle}{\emph{15\textsuperscript{th}
  International Conference on Automation Science and Engineering}}.
\newblock


\bibitem[Kim and Nelson(2007)]%
        {Kim2007RSoverview}
\bibfield{author}{\bibinfo{person}{Seong-Hee Kim} {and}
  \bibinfo{person}{Barry~L. Nelson}.} \bibinfo{year}{2007}\natexlab{}.
\newblock \showarticletitle{Recent Advances in Ranking and Selection}. In
  \bibinfo{booktitle}{\emph{Proceedings of the 2007 Winter Simulation
  Conference}}. \bibinfo{publisher}{Institute of Electrical and Electronics
  Engineers, Inc}, \bibinfo{address}{Piscataway, NJ},
  \bibinfo{pages}{162–172}.
\newblock


\bibitem[Li et~al\mbox{.}(2020)]%
        {li2020contextdependent}
\bibfield{author}{\bibinfo{person}{Haidong Li}, \bibinfo{person}{Henry Lam},
  \bibinfo{person}{Zhe Liang}, {and} \bibinfo{person}{Yijie Peng}.}
  \bibinfo{year}{2020}\natexlab{}.
\newblock \showarticletitle{Context-Dependent Ranking and Selection under a
  Bayesian Framework}. In \bibinfo{booktitle}{\emph{Proceedings of the 2020
  Winter Simulation Conference}}. \bibinfo{publisher}{Institute of Electrical
  and Electronics Engineers, Inc}, \bibinfo{address}{Piscataway, NJ},
  \bibinfo{pages}{2060--2070}.
\newblock


\bibitem[Li and Gao(2021)]%
        {li-gao2021convergence}
\bibfield{author}{\bibinfo{person}{Yanwen Li} {and} \bibinfo{person}{Siyang
  Gao}.} \bibinfo{year}{2021}\natexlab{}.
\newblock \showarticletitle{On the Convergence of Optimal Computing Budget
  Allocation Algorithms}. In \bibinfo{booktitle}{\emph{Proceedings of the 2021
  Winter Simulation Conference}}. \bibinfo{publisher}{Institute of Electrical
  and Electronics Engineers, Inc}, \bibinfo{address}{Piscataway, NJ}.
\newblock


\bibitem[Meyer(2000)]%
        {Meyer2000MatrixAnalysis}
\bibfield{author}{\bibinfo{person}{Carl~D. Meyer}.}
  \bibinfo{year}{2000}\natexlab{}.
\newblock \bibinfo{booktitle}{\emph{Matrix Analysis and Applied Linear
  Algebra}}.
\newblock \bibinfo{publisher}{Society for Industrial and Applied Mathematics},
  \bibinfo{address}{Philadelphia, PA}.
\newblock


\bibitem[Nunes and Hu(2012)]%
        {Nunes2012Recommender}
\bibfield{author}{\bibinfo{person}{Maria Augusta~S.N. Nunes} {and}
  \bibinfo{person}{Rong Hu}.} \bibinfo{year}{2012}\natexlab{}.
\newblock \showarticletitle{Personality-Based Recommender Systems: An
  Overview}. In \bibinfo{booktitle}{\emph{Proceedings of the Sixth ACM
  Conference on Recommender Systems}}.
\newblock


\bibitem[Pearce and Branke(2018)]%
        {PEARCE2018IKG-REVI}
\bibfield{author}{\bibinfo{person}{Michael Pearce} {and}
  \bibinfo{person}{Juergen Branke}.} \bibinfo{year}{2018}\natexlab{}.
\newblock \showarticletitle{Continuous Multi-Task Bayesian Optimisation with
  Correlation}.
\newblock \bibinfo{journal}{\emph{European Journal of Operational Research}}
  \bibinfo{volume}{270}, \bibinfo{number}{3} (\bibinfo{year}{2018}),
  \bibinfo{pages}{1074 -- 1085}.
\newblock


\bibitem[Pearce et~al\mbox{.}(2020)]%
        {pearce2020IKG-ConBO}
\bibfield{author}{\bibinfo{person}{Michael Pearce}, \bibinfo{person}{Janis
  Klaise}, {and} \bibinfo{person}{Matthew Groves}.}
  \bibinfo{year}{2020}\natexlab{}.
\newblock \showarticletitle{Practical Bayesian Optimization of Objectives with
  Conditioning Variables}.
\newblock  (\bibinfo{year}{2020}).
\newblock
\showeprint[arxiv]{2002.09996}


\bibitem[Shen et~al\mbox{.}(2021)]%
        {shen2021ranking}
\bibfield{author}{\bibinfo{person}{Haihui Shen}, \bibinfo{person}{L.~Jeff
  Hong}, {and} \bibinfo{person}{Xiaowei Zhang}.}
  \bibinfo{year}{2021}\natexlab{}.
\newblock \showarticletitle{Ranking and Selection with Covariates for
  Personalized Decision Making}.
\newblock \bibinfo{journal}{\emph{INFORMS Journal on Computing}}
  (\bibinfo{year}{2021}).
\newblock


\bibitem[Surjanovic and Bingham(2013)]%
        {test_functions}
\bibfield{author}{\bibinfo{person}{S. Surjanovic} {and} \bibinfo{person}{D.
  Bingham}.} \bibinfo{year}{2013}\natexlab{}.
\newblock \showarticletitle{Hartmann 3-Dimensional Function}. In
  \bibinfo{booktitle}{\emph{Virtual Library of Simulation Experiments}}.
\newblock
\urldef\tempurl%
\url{https://www.sfu.ca/~ssurjano/hart3.html}
\showURL{%
\tempurl}


\end{thebibliography}

\appendix

\section{Proof of Theorem \ref{thm-pcs-rate-function}}
With $n$ denoting the total number of observations / samples, let $p(k, c)$ and $N_n(k, c) = n p(k, c)$ denote the fraction and the total number, respectively, of samples allocated to $(k, c)$. 
Both $p(k, c)$ and $N(k, c)$ are determined by the sampling policy and are assumed to be strictly positive. We ignore the technicalities arising from $N(k, c)$ not being an integer. For sake of simplicity, we leave implicit the dependency of the probabilities and other quantities on the sampling policy. We start with the following lemma, which is crucial in proving the theorem.

\begin{lemma} \label{lemma-rate-func-term}
    Under assumptions of Theorem \ref{thm-pcs-rate-function}, for $k \neq \pi^*(c)$,
    \begin{equation*}
        \lim_{n \rightarrow \infty} \frac{1}{n} \log P(\mu_n(\pi^*(c), c) < \mu_n(k, c)) = -G_{(k, c)} (p(\pi^*(c), c), p(k, c)),
    \end{equation*}
    where 
    \begin{equation*}
        G_{(k, c)} (p({\pi^*(c)}, c), p(k, c)) = \frac{(F({\pi^*(c)}, c) - F(k, c))^2}{2(\sigma^2({\pi^*(c)}, c) / p({\pi^*(c)}, c) + \sigma^2(k, c) / p(k, c))}.
    \end{equation*}
\end{lemma}
\begin{proof}
    We will follow the analysis of \cite{Glynn2004LargeDev} and use the Gartner-Ellis Theorem \cite{dembo1998Largedev} to find $G_{(k, c)} (p(\pi^*(c), c), p(k, c))$, which requires understanding the distributional behavior of $\mu_n(k, c)$. In particular, we need to study the limiting behavior of the log moment generating function (MGF):
    \begin{equation*}
        \Lambda_n(\lambda; k, c) = \log \E[\exp (\lambda \mu_n(k, c)) ].
    \end{equation*}
    
    Using the conjugacy property of GPs (under the assumption of Gaussian observation noise with known variance) and updating the posterior using samples from one context at a time, we can decompose $\mu_n(k, c)$ as
    \begin{equation*}
        \mu_n(k, c) = \mu_0(k, c) + \sum_{i=1}^{|\mathcal{C}|} [\Sigma_n^{i-1}(c, c_i; k)]_{1 \times N_n(k, c_i)} (A^{i})^{-1} (Y_n(k, c_i) - [\mu_n^{i-1}(k, c_i)]_{N_n(k, c_i) \times 1}),
    \end{equation*}
    where $\mu_n^{i-1}(k, c_i)$ is defined in the same way except with the summation being from $1$ to $i-1$ with $\mu_n^0(\cdot, \cdot) = \mu_0(\cdot, \cdot)$,
    $[\alpha]_{j \times m}$ denotes the $j \times m$ matrix where each element is $\alpha$, \linebreak
        $A^{i} = [\Sigma_n^{i-1}(c_i, c_i; k)]_{N_n(k, c_i) \times N_n(k, c_i)} + diag_{N_n(k, c_i)}(\sigma^2(k, c_i))$,
    with $diag_{N}(\beta)$ denoting the diagonal matrix of size $N \times N$ with diagonals $\beta$, $Y_n(k, c_i)$ denotes the $N_n(k, c_i) \times 1$ matrix of observations corresponding to $k, c_i$, and 
    \begin{equation*}
        \Sigma_n^{i}(c, c'; k) = \Sigma_n^{i-1}(c, c'; k) - [\Sigma_n^{i-1}(c, c_{i}; k)]_{1 \times N_n(k, c_{i})} (A^{i})^{-1} [\Sigma_n^{i-1}(c_{i}, c'; k)]_{N_n(k, c_{i}) \times 1},
    \end{equation*}
    with $\Sigma_n^0(\cdot, \cdot; k) = \Sigma_0(\cdot, \cdot; k)$. The inverse of $A^i$ can be calculated in closed form using the Sherman-Morrison formula \cite{Meyer2000MatrixAnalysis}. After some algebra,
    we can rewrite $\mu_n(k, c)$ as follows:
    \begin{equation}\label{eq-mean-breakdown}
        \mu_n(k, c) = \mu_0(k, c) + \sum_{i=1}^{|\mathcal{C}|} \frac{N_n(k, c_i) (\widebar{Y}_n(k, c_i) - \mu_n^{i-1}(k, c_i)) \Sigma_n^{i-1}(c, c_i; k)}{\sigma^2(k, c_i) + N_n(k, c_i) \Sigma_n^{i-1}(c_i, c_i; k)},
    \end{equation}
    where $\widebar{Y}_n(k, c_i)$ denotes the average of the observations. Similarly, we can rewrite $\Sigma_n^{i}(c, c'; k)$ as:
    \begin{equation*}
        \Sigma_n^i(c, c'; k) = \Sigma_n^{i-1}(c, c'; k) - \frac{N_n(k, c_{i}) \Sigma_n^{i-1}(c, c_i; k) \Sigma_n^{i-1}(c_i, c'; k)}{\sigma^2(k, c_i) + N_n(k, c_i) \Sigma_n^{i-1}(c_i, c_i; k)}.
    \end{equation*}
    
    For a Gaussian random variable $\mathcal{N}(\Tilde{\mu}, \Tilde{\sigma}^2)$, the log-MGF is given by $\Tilde{\mu} \lambda + \Tilde{\sigma}^2 \lambda^2 / 2$. Since the true distribution of samples is $y(k, c) \sim \mathcal{N}(F(k, c), \sigma^2(k, c))$
    and the samples are independent of each other, we can view $\mu_n(\cdot, \cdot)$ as a linear combination of independent Gaussian random variables and write the log-MGF
    \begin{equation*}
        \Lambda_n(\lambda; k, c) = \mu_0(k, c) \lambda + \sum_{i=1}^{|\mathcal{C}|} \left[ (F(k, c_i) - \mu_n^{i-1}(k, c_i))C_n(k, c, i) \lambda + \frac{\sigma^2(k, c_i) C_n(k, c, i)^2 \lambda^2}{2 N_n(k, c_i)} \right],
    \end{equation*}
    where
    \begin{equation*}
        C_n(k, c, i) = \frac{N_n(k, c_i) \Sigma_n^{i-1}(c, c_i; k)}{\sigma^2(k, c_i) + N_n(k, c_i) \Sigma_n^{i-1}(c_i, c_i; k)}.
    \end{equation*}
    Let $\Lambda_n(\lambda_{\pi^*(c)}, \lambda_k; c)$ denote the log-MGF of $Z_n = (\mu_n(\pi^*(c), c), \mu_n(k, c))$.
    In order to use the Gartner-Ellis Theorem, we need to establish the limiting behavior of $\frac{1}{n} \Lambda_n(n \lambda_{\pi^*(c)}, n \lambda_k; c)$.
    \begin{equation}\label{eq-limiting-rate-part-1}
        \lim_{n \rightarrow \infty} \frac{1}{n} \Lambda_n(n \lambda_{\pi^*(c)}, n \lambda_k; c) = \sum_{\kappa \in (\pi^*(c), k)} \lim_{n \rightarrow \infty} \E[\mu_n(\kappa, c)] \lambda_{\kappa} + \lim_{n \rightarrow \infty} \frac{n Var(\mu_n(\kappa, c)) \lambda_{\kappa}^2}{2}.
    \end{equation}
    
    Let us start with the variance term. In the following, we use $\rightarrow$ to denote the limit as $n \rightarrow \infty$ and $\xrightarrow{\approx}$ to denote equivalence in the limit. Note that
    \begin{equation*}
        Var(\mu_n(k, c)) = \sum_{i=1}^{|\mathcal{C}|} \frac{\sigma^2(k, c_i) C_n(k, c, i)^2}{N_n(k, c_i)}.
    \end{equation*}
    Due to conjugacy of GPs, we can choose to process the summation in any order, as long as we follow the same order for updating $\Sigma_n^i(\cdot, \cdot; k)$. Let us analyze $Var(\mu_n({\pi^*(c)}, c))$, for a given $c$, with the summation and the update processed starting from $c$, i.e., using $c_1 = c$ with an appropriate re-ordering of $\mathcal{C}$. Note that
    \begin{equation*}
        \Sigma_n^i(c', c''; k) \rightarrow \Sigma_n^{i-1}(c', c''; k) - \frac{\Sigma_n^{i-1}(c', c_i; k) \Sigma_n^{i-1}(c_i, c''; k)}{\Sigma_n^{i-1}(c_i, c_i; k)},
    \end{equation*}
    which implies that $\Sigma_n^i(\cdot, c_1; k) \rightarrow 0, i \geq 1$, and $C_n(k, c', i) \rightarrow \frac{\Sigma_0(c', c_i; k)}{\Sigma_0(c_i, c_i; k)}$ if $i=1$ and $C_n(k, c', i) \rightarrow 0$ otherwise. Thus, we can ignore the rest of the terms in the summation and write
    \begin{equation*}
        Var(\mu_n(k, c)) \xrightarrow{\approx} \frac{\sigma^2(k, c)}{N_n(k, c)} = \frac{\sigma^2(k, c)}{n p(k, c)}. 
    \end{equation*}
    
    Similarly for the expectation term, using $c_1 = c$, since $\Sigma_n^i(\cdot, c_1; k) \rightarrow 0, i \geq 1$, in the limit equation (\ref{eq-mean-breakdown}) becomes
    \begin{equation*}
        \mu_n(k, c) \xrightarrow{\approx} \mu_0(k, c) + (\widebar{Y}_n(k, c_1) - \mu_n^0(k, c_1)) = \widebar{Y}_n(k, c),
    \end{equation*}
    which implies that $\lim_{n \rightarrow \infty} \E[\mu_n(k, c)] = F(k, c)$.
    We are now ready to continue from (\ref{eq-limiting-rate-part-1}). Let $\Lambda^t(\lambda_k; k, c)$ denote the log-MGF of the observation $y(k, c) \sim \mathcal{N}(F(k, c), \sigma^2(k, c))$.
    \begin{equation*}
        \begin{aligned}
            \lim_{n \rightarrow \infty} & \frac{1}{n} \Lambda_n(n \lambda_{\pi^*(c)}, n \lambda_k) = \sum_{\kappa \in (\pi^*(c), k)} F(\kappa, c) \lambda_{\kappa} + \frac{\sigma^2(\kappa, c) \lambda_{\kappa}^2}{2 p(\kappa, c)} \\
            &= \sum_{\kappa \in (\pi^*(c), k)} p(\kappa, c) \left( \frac{F(\kappa, c) \lambda_{\kappa}}{p(\kappa, c)} + \frac{\sigma^2(\kappa, c) \lambda_{\kappa}^2}{2 p(\kappa, c)^2} \right) = \sum_{\kappa \in (\pi^*(c), k)} p(\kappa, c) \Lambda^t(\lambda_{\kappa} / p(\kappa, c); \kappa, c),
        \end{aligned}
    \end{equation*}
    which is the exact term in Lemma 1 of \cite{Glynn2004LargeDev}. Following the steps therein, we find that the rate function of $Z_n$ is given by
    \begin{equation*}
        I(x_{\pi^*(c)}, x_k) = p({\pi^*(c)}, c) I^{t}(x_{\pi^*(c)}; {\pi^*(c)}, c) + p(k, c) I^{t}(x_k; k, c),
    \end{equation*}
    where $I^{t}(x_k; k, c) = \frac{(x_k - F(k, c))^2}{2 \sigma^2(k, c)}$ is the Fenchel-Legendre transform of $\Lambda^t(\lambda_k / p(k, c); k, c)$.
    With the rate function of $Z_n$ established,
    \cite{Glynn2004LargeDev} shows that
    \begin{equation*}
        G_{(k, c)} (p({\pi^*(c)}, c), p(k, c)) = \inf_{x_{{\pi^*(c)}} \geq x_k} \left[ p({\pi^*(c)}, c) I^{t}(x_{\pi^*(c)}; {\pi^*(c)}, c) + p(k, c) I^{t}(x_k; k, c) \right],
    \end{equation*}
    where the infimum can be calculated via differentiation \cite{Gao2019Covariates}, giving us
    \begin{equation*}
        G_{(k, c)} (p({\pi^*(c)}, c), p(k, c)) = \frac{(F({\pi^*(c)}, c) - F(k, c))^2}{2(\sigma^2({\pi^*(c)}, c) / p({\pi^*(c)}, c) + \sigma^2(k, c) / p(k, c))}.
    \end{equation*}
\end{proof}


With the lemma established, the theorem can be proved as follows.
\begin{proof}[Proof of Theorem \ref{thm-pcs-rate-function}]
For a given context $c$, the probability of false selection after $n$ observations $PFS^n(c) = 1 - PCS^n(c)$ is given by 
\begin{equation*}
    PFS^n(c) = P(\mu_n(\pi^*(c), c) < \mu_n(k, c), \exists k \neq \pi^*(c)).
\end{equation*}
We can lower and upper bound this respectively by
\begin{equation*}
    \max_{k \neq \pi^*(c)} P(\mu_n(\pi^*(c), c) < \mu_n(k, c)) \text{ and } (|\mathcal{K}| - 1) \max_{k \neq \pi^*(c)} P(\mu_n(\pi^*(c), c) < \mu_n(k, c)).
\end{equation*}
If for $k \neq \pi^*(c)$, 
\begin{equation*}
    \lim_{n \rightarrow \infty} \frac{1}{n} \log P(\mu_n(\pi^*(c), c) < \mu_n(k, c)) = -G_{(k, c)} (p(\pi^*(c), c), p(k, c))
\end{equation*}
for some rate function $G_{(k, c)}$, then
\begin{equation*}
    \lim_{n \rightarrow \infty} \frac{1}{n} \log PFS^n(c) = - \min_{k \neq \pi^*(c)} G_{(k, c)} (p(\pi^*(c), c), p(k, c)).
\end{equation*}
Similarly, both the $PFS^n_E$ and $PFS^n_M$ can be lower and upper bounded by $\max_c PFS^n(c)$ and $(|\mathcal{K}| - 1) \max_c PFS^n(c)$ (or with an additional constant factor $\max_c w(c) / \min_c w(c)$ if $w(c)$ are not uniform), respectively. Thus, we can extend this to write the rate function of contextual $PCS$ as
\begin{equation*}
    \lim_{n\rightarrow\infty} \frac{1}{n} \log PFS^n_{\sim} = - \min_{c \in \mathcal{C}} \min_{k \neq \pi^*(c)} G_{(k, c)} (p(\pi^*(c), c), p(k, c)),
\end{equation*}
where $PFS^n_{\sim}$ is either of $PFS^n_E$ or $PFS^n_M$. Combining this with Lemma \ref{lemma-rate-func-term}, we get the result.
\begin{equation*}
    \lim_{n\rightarrow\infty} \frac{1}{n} \log PFS^n_{\sim} = - \min_{c \in \mathcal{C}} \min_{k \neq \pi^*(c)} \frac{(F({\pi^*(c)}, c) - F(k, c))^2}{2(\sigma^2({\pi^*(c)}, c) / p({\pi^*(c)}, c) + \sigma^2(k, c) / p(k, c))}.
\end{equation*}
\end{proof}

\section{Proofs of Lemmas and Theorem \ref{theorem-infinitely-sample}}
Throughout this section, we use the notation $N_{\infty}(k, c) = \lim_{n \to \infty} N_n(k, c)$. 
We use $\pi^{\infty}(c)$ to denote the best alternative reported by the algorithm as the number of samples goes to infinity, whose existence follows from the existence of the limiting GP given by the Proposition 2.9 of \cite{bect2019GPSupermartingale}.
In addition, we alter the definition of $\hat{p}_n(k, c)$, without affecting the algorithm behavior, to be the fraction of samples allocated to alternative $k$ under context $c$ as the fraction of total samples allocated to context $c$, i.e., $\hat{p}_n(k, c) = \frac{N_n(k, c)}{\sum_{k'}N_n(k', c)}$, which differs from the original definition in that the denominator lacks the summation over contexts. This ensures that at least one of $\psi^{(1)}(c)$ and $\psi^{(2)}(c)$ (defined in Algorithm \ref{alg:gp-c-ocba}) is always strictly positive and helps avoid technical difficulties in proving the results.

\subsection{Proof of Lemma \ref{lemma-posterior-correlation}}
\begin{proof}
    We will show this by contradiction. It is trivial to show that if $Corr_n(c, c'; k) = \pm 1$, then $Corr_{n'}(c, c'; k) = \pm 1$ for all $n' \geq n$. Thus, it is sufficient to show that we cannot have $\lim_{n \rightarrow \infty} Corr_n(c, c'; k) = \pm 1$.
    
    Suppose that we have $n_e > 0$ samples from all $c \in \mathcal{C}$. Note that this is without loss of generality since we can always add more samples to ensure that each context receives the same number of samples (by the first argument, this will not change the result). 
    We can group the samples into $n_e$ vector observations where each vector follows multivariate normal distribution with covariance matrix $diag([\sigma^2(k, c)]_{c \in \mathcal{C}})$. 
    The crucial realization at this step is that, given only full vector observations, the GP posterior update formulas are equivalent to the well known formulas for the posterior of a multivariate normal distribution with known covariance, where the prior for the unknown mean is another multivariate normal distribution. That is, 
    \begin{align*}
        \Sigma_{n_e |\mathcal{C}|}(\mathcal{C}, \mathcal{C}; k) &= \left( \Sigma_0(\mathcal{C}, \mathcal{C}; k)^{-1} + n_e  diag([\sigma^2(k, c)]_{c \in \mathcal{C}})^{-1} \right)^{-1} \\
        &\xrightarrow{\approx} \frac{1}{n_e} diag([\sigma^2(k, c)]_{c \in \mathcal{C}}).
    \end{align*}
    This implies that $Corr_{n_e |\mathcal{C}|}(c, c'; k) \rightarrow 0$ as $n_e \rightarrow \infty$ for all $c \neq c'$. Thus, it must be the case that $-1 < Corr_n(c, c'; k) < 1, \forall n \geq 0, c \neq c'$.
\end{proof}

\subsection{Proof of Lemma \ref{lemma-posterior-variance}}
\begin{proof}
    Proposition 2.9 of \cite{bect2019GPSupermartingale} shows that for any sequential design strategy, the GP posterior converges almost surely to a limiting distribution, which is also a GP. 
    Let $\Sigma_{\infty}(k, c)$ denote the variance and $\Sigma_{\infty}(c, c'; k)$ the covariance of the limiting GP. Recall from the proof of Theorem \ref{thm-pcs-rate-function} that
    \begin{equation*}
        \Sigma_n(c, c'; k) = \Sigma_0(c, c'; k) - \sum_{i = 1}^{|\mathcal{C}|} \frac{N_n(k, c_i) \Sigma^{i-1}_n(c, c_i; k) \Sigma^{i-1}_n(c_i, c'; k)}{\sigma^2(k, c_i) + N_n(k, c_i) \Sigma^{i-1}_n(c_i, c_i; k)};
    \end{equation*}
    where
    \begin{equation*}
        \Sigma^{i}_n(c, c'; k) = \Sigma^{i-1}_n(c, c'; k) - \frac{N_n(k, c_i) \Sigma^{i-1}_n(c, c_i; k) \Sigma^{i-1}_n(c_i, c'; k)}{\sigma^2(k, c_i) + N_n(k, c_i) \Sigma^{i-1}_n(c_i, c_i; k)}.
    \end{equation*}
    It is easy to see that if $N_n(k, c_i) \to \infty$,
    \begin{equation*}
    \begin{aligned}
        \Sigma^i_n(c_i, c; k) &\to \Sigma^{i-1}_n(c_i, c; k) - \frac{\Sigma^{i-1}_n(c_i, c_i; k) \Sigma^{i-1}_n(c_i, c; k)}{\Sigma^{i-1}_n(c_i, c_i; k)} \\
        &\qquad= \Sigma^{i-1}_n(c_i, c; k) - \Sigma^{i-1}_n(c_i, c; k) = 0.
    \end{aligned}
    \end{equation*}
    Thus, $\Sigma_n(c_i, c; k) \to 0, \forall c \in \mathcal{C}$ and $\Sigma_{\infty}(k, c_i) = \lim_{n \to \infty} \Sigma_n(c_i, c_i; k) = 0$, if $N_n(k, c_i) \to \infty$. Note that the above convergence is only valid if $\Sigma^{i-1}_n(c_i, c; k) > 0$. The other alternative is $\Sigma^{i-1}_n(c_i, c; k) = 0$, in which case the same result trivially holds.

    Suppose that $(k, c')$ gets sampled only finitely often, i.e., $N_{\infty}(k, c') < \infty$.
    Noting that $\Sigma_n(c, c'; k) = \Sigma_n^{|\mathcal{C}|}(c, c'; k)$, it is sufficient to show that $\lim_{n \rightarrow \infty} \Sigma^i_n(c', c'; k) > 0, \forall i = 1, \ldots, |\mathcal{C}|$. We will show this via induction. For the base case, we have $\lim_{n \rightarrow \infty} \Sigma^0_n(c', c'; k) = \Sigma_0(k, c') > 0$. Using $\Sigma^i_{\infty}(\cdot, \cdot; \cdot) := \lim_{n \rightarrow \infty} \Sigma^i_n(\cdot, \cdot; \cdot)$, suppose that $\Sigma^{i-1}_{\infty}(c', c'; k) > 0$, then
    \begin{equation*}
    \begin{aligned}
        \Sigma^{i}_{\infty}&(c', c'; k) = \Sigma^{i-1}_{\infty}(c', c'; k) - \frac{N_{\infty}(k, c_i) \Sigma^{i-1}_{\infty}(c', c_i; k) \Sigma^{i-1}_{\infty}(c_i, c'; k)}{\sigma^2(k, c_i) + N_{\infty}(k, c_i) \Sigma^{i-1}_{\infty}(c_i, c_i; k)} \\
        &= \frac{ \Sigma^{i-1}_{\infty}(c', c'; k) \sigma^2(k, c_i) + N_{\infty}(k, c_i) \left( \Sigma^{i-1}_{\infty}(c', c'; k) \Sigma^{i-1}_{\infty}(c_i, c_i; k) - \Sigma^{i-1}_{\infty}(c', c_i; k) \Sigma^{i-1}_{\infty}(c_i, c'; k) \right)}{ \sigma^2(k, c_i) + N_{\infty}(k, c_i) \Sigma^{i-1}_{\infty}(c_i, c_i; k) }.
    \end{aligned}
    \end{equation*}
    Note that $\Sigma^{i}_n(\cdot, \cdot; \cdot)$ denotes the conditional covariance with conditioning on a certain subset of the observations. Thus, $|\Sigma^{i-1}_n(c', c_i; k)| \leq \sqrt{\Sigma^{i-1}_n(c', c'; k) \Sigma^{i-1}_n(c_i, c_i; k)}$, and it follows that if $N_{\infty}(k, c_i) < \infty$,
    \begin{equation*}
        \Sigma^{i}_{\infty}(c', c'; k) \geq \frac{\Sigma^{i-1}_{\infty}(c', c'; k) \sigma^2(k, c_i)}{\sigma^2(k, c_i) + N_{\infty}(k, c_i) \Sigma^{i-1}_{\infty}(c_i, c_i; k)} > 0.
    \end{equation*}
    We need to also show the case where $N_{\infty}(k, c_i) = \infty$. Note that the result we showed at the first part is $\Sigma_{\infty}^i(c_i, c'; k) = 0$, not $\Sigma_{\infty}^{i-1}(c_i, c'; k) = 0$, which we need to work with here. If $\Sigma_{\infty}^{i-1}(c_i, c'; k) = 0$, the result follows immediately. Assuming $\Sigma_{\infty}^{i-1}(c_i, c'; k) > 0$,
    \begin{align*}
        \Sigma_{\infty}^i (c', c'; k) &=  \frac{ \Sigma^{i-1}_{\infty}(c', c'; k) \Sigma^{i-1}_{\infty}(c_i, c_i; k) - \Sigma^{i-1}_{\infty}(c', c_i; k) \Sigma^{i-1}_{\infty}(c_i, c'; k) }{ \Sigma^{i-1}_{\infty}(c_i, c_i; k) }.
    \end{align*}
    If $|\Sigma^{i-1}_n(c', c_i; k)| < \sqrt{\Sigma^{i-1}_n(c', c'; k) \Sigma^{i-1}_n(c_i, c_i; k)}$, $\Sigma_{\infty}^i (c', c'; k)$ above is strictly positive as needed. Note that this is given by Lemma \ref{lemma-posterior-correlation}. Thus, the result holds.
\end{proof}

\subsection{Proof of Lemma \ref{lemma-single-context-infinite-samples}}
\begin{proof}
    We will prove this in two parts. First, we will show that the best predicted alternative $\pi^{\infty}(c)$ must get sampled infinitely often. We will then show that all other alternatives must also get sampled infinitely often. The proof relies heavily on Lemma \ref{lemma-posterior-variance}, which ensures the existence of the scenarios constructed below.
    
    Suppose that $(\pi^{\infty}(c), c)$ does not get sampled infinitely often, i.e., $N_{\infty}(\pi^{\infty}(c), c) < \infty$. Let $k' \neq \pi^{\infty}(c)$ denote an alternative such that $N_{\infty}(k', c) = \infty$, which must exist since $\sum_{k \in \K} N_{\infty}(k, c) = \infty$. 
    There exists some $M < \infty$, after which $(\pi^{\infty}(c), c)$ no longer gets sampled, some $0 < \delta < 1$, and some $n > M$ where the algorithm chooses to sample $(k', c)$ such that
    \begin{itemize}
        \item $\pi^n(c) = \pi^\infty(c)$;
        \item $\hat{p}_n(\pi^{\infty}(c), c) < \delta$ and $\Sigma_n(\pi^{\infty}(c), c) \geq \Sigma_{\infty}(\pi^{\infty}(c), c) > \delta$;
        \item $\hat{p}_n(k', c) > \delta$ and $\Sigma_n(k', c) < \delta$.
    \end{itemize}
    Note that in this case, we have $\psi^{(1)}(c) < 1$ and $\psi^{(2)}(c) > 1$. Thus $\psi^{(1)}(c) < \psi^{(2)}(c)$ and the algorithm must sample from $(\pi^{\infty}(c), c)$ instead, leading to a contradiction.
    
    
    Suppose that the opposite happens, i.e., only the predicted best alternative $(\pi^{\infty}(c), c)$ receives infinitely many samples (among all the alternatives for context $c$). In this case, we would have $\psi^{(1)}(c) \rightarrow \infty$ and $\psi^{(2)}(c) \rightarrow 0$, which would also lead to a contradiction. More precisely, there exists some $M < \infty$ such that after iteration $M$, among all the pairs with context $c$, only $(\pi^{\infty}(c), c)$ receives samples. Let $0 < \delta < 1 / |\mathcal{C}|$ be a constant, and $n > M$ be an iteration where context $c$ (or the pair $(\pi^{\infty}(c), c)$) receives another sample such that
    \begin{itemize}
        \item $\pi^n(c) = \pi^{\infty}(c)$;
        \item $\hat{p}_n(\pi^{\infty}(c), c) > |\mathcal{C}| \delta$ and $\Sigma_n(\pi^{\infty}(c), c) < \delta$;
        \item $\hat{p}_n(k, c) < \delta$ and $\Sigma_n(k, c) > \delta$ for all $c \neq \pi^{\infty}(c)$.
    \end{itemize}
    In this case, $\psi^{(1)}(c) > |\mathcal{C}|$ and $\psi^{(2)}(c) < |\mathcal{C}| - 1$. Thus, $\psi^{(1)}(c) > \psi^{(2)}(c)$, which violates the condition for $(\pi^{\infty}(c), c)$ to get sampled from. We have a contradiction and some other alternative for context $c$ must also get sampled infinitely often.
    
    
    Now, suppose that a certain subset of non-optimal alternatives does not get sampled infinitely often, denoted by $\K^{fin} := \{k \neq \pi^{\infty}(c): N_{\infty}(k, c) < \infty \}$. 
    By Proposition 2.9 of \cite{bect2019GPSupermartingale}, the GP posterior converges almost surely to a limiting GP. Denoting the limiting posterior mean by $\mu_{\infty}(k, c)$, for any $\delta > 0$ there exists some $M < \infty$ such that
    \begin{equation*}
        P(\mu_{\infty}(k, c) - \delta < \mu_n(k, c) < \mu_{\infty}(k, c) + \delta, \forall k \in \K, n > M) = 1.
    \end{equation*}
    Let $d_{min} = \min_{k \neq \pi^{*}(c)} \mu_{\infty}(\pi^{*}(c), c) - \mu_{\infty}(k, c) - 2 \delta$ and $d_{max} = \max_{k \neq \pi{*}(c)} \mu_{\infty}(\pi^{*}(c), c) - \mu_{\infty}(k, c) + 2 \delta$, and note that $d_{min} < \mu_n(\pi^n(c), c) - \mu_n(k, c) < d_{max}$ for all $k \neq \pi^n(c)$ almost surely for all $n > M$. Let $\kappa = 2 * (d_{max} / d_{min})^2$. There exists some $\delta > 0$ and $n > M$ where the algorithm samples from context $c$ such that
    \begin{itemize}
        \item $\pi^{n}(c) = \pi^{\infty}(c)$;
        \item $\psi^{(1)}(c) > \psi{(2)}(c)$;
        \item $\Sigma_n(k, c) > \delta$ for all $k \in \K'$;
        \item $\Sigma_n(k, c) < \delta / \kappa$ for all $k \notin \K'$.
    \end{itemize}
    In this case, for $k \notin \K', k \neq \pi^n(c)$, we have
    \begin{equation*}
        \zeta(k, c) = \frac{(\mu_n(\pi^n(c), c) - \mu_n(k, c))^2}{\Sigma_n(\pi^n(c), c) + \Sigma_n(k, c)} > \frac{d_{min}^2}{\delta / (d_{\max} / d_{\min})^2} = \frac{d_{\max}^2}{\delta},
    \end{equation*}
    and for $k \in \K'$, we have
    \begin{equation*}
        \zeta(k, c) = \frac{(\mu_n(\pi^n(c), c) - \mu_n(k, c))^2}{\Sigma_n(\pi^n(c), c) + \Sigma_n(k, c)} < \frac{d_{\max}^2}{\delta}
    \end{equation*}
    almost surely. Thus, the minimizer of $\zeta(k, c)$ and the next alternative to get sampled must belong to $\K'$, which leads to a contradiction. 
    
    We have established that if a context gets sampled from infinitely often, then best predicted alternative must get sampled infinitely often, some non-best predicted alternative must get sampled infinitely often, and all other alternatives must get sampled from infinitely often almost surely. Thus, all alternatives for this context must get sampled from infinitely often almost surely, concluding the proof.
\end{proof}

\subsection{Proof of Theorem \ref{theorem-infinitely-sample}}
\begin{proof}
    We will follow similar arguments to Lemma \ref{lemma-single-context-infinite-samples} to prove this result. Since $\sum_{k, c} N_{\infty}(k, c) = \infty$, there must be some contexts that have been sampled from infinitely often. Let $\mathcal{C}^{fin} := \{c \in \mathcal{C}: \sum_{k} N_{\infty}(k, c) < \infty\}$ and $\mathcal{C}^{inf} := \mathcal{C} \setminus \mathcal{C}^{fin}$. By Lemma \ref{lemma-single-context-infinite-samples}, we know that $N_{\infty}(k, c) = \infty$ for all $k \in \K, c \in \mathcal{C}^{inf}$ almost surely.
    
    Following the arguments in Lemma \ref{lemma-single-context-infinite-samples}, for some $\delta > 0$, let $d_{min} = \min_{k \neq \pi^{*}(c), c \in \mathcal{C}} \mu_{\infty}(\pi^{*}(c), c) - \mu_{\infty}(k, c) - 2 \delta$ and $d_{max} = \max_{k \neq \pi{*}(c), c \in \mathcal{C}} \mu_{\infty}(\pi^{*}(c), c) - \mu_{\infty}(k, c) + 2 \delta$. Note that $d_{min} < \mu_n(\pi^n(c), c) - \mu_n(k, c) < d_{max}$ for all $k \neq \pi^n(c), c \in \mathcal{C}$ almost surely for all $n > M$, and some $M < \infty$. Pick $M$ large enough so that no $c \in \mathcal{C}^{fin}$ gets sampled from after iteration $M$. Let  $\kappa = 2 * (d_{max} / d_{min})^2$.
    Pick some $\delta > 0$ and $n > M$ (where some $c \in \mathcal{C}^{inf}$ gets sampled from) such that
    \begin{itemize}
        \item $\Sigma_n(k, c) > \delta / 2$ for all $k \in \K, c \in \mathcal{C}^{fin}$;
        \item $\Sigma_n(k, c) < \delta / \kappa$ for all $k \in \K, c \in \mathcal{C}^{inf}$.
    \end{itemize}
    In this case, we have that $\zeta(k, c) > \frac{d_{max}^2}{\delta}$ for $k \neq \pi^{n}(c), c \in \mathcal{C}^{inf}$ and $\zeta(k, c) < \frac{d_{max}^2}{\delta}$ for $k \neq \pi^{n}(c), c \in \mathcal{C}^{fin}$. Thus, the context minimizing $\zeta(k, c)$ is in $\mathcal{C}^{fin}$ almost surely, leading to a contradiction. We can conclude that GP-C-OCBA must sample from each alternative-context pair infinitely often, almost surely.
\end{proof}

\subsection{Proof of Lemma \ref{lemma-posterior-mean-consistency}}
\begin{proof}
    Recall from derivation of the rate function that we can process the GP updates in any particular order thanks to the conjugacy of GPs. Repeating the steps therein with the updates processed starting with $(k, c)$, i.e. using $c_1 = c$, we get that
    \begin{equation*}
    \begin{aligned}
        \lim_{n \rightarrow \infty} \mu_{n}(k, c) &= \mu_0(k, c) + \frac{N_{\infty}(k, c) (\widebar{Y}_{\infty}(k, c) - \mu_0(k, c)) \Sigma_0(k, c)}{\sigma^2(k, c) + N_{\infty}(k, c) \Sigma_0(k, c)} \\
        &= \mu_0(k, c) + (\widebar{Y}_{\infty}(k, c) - \mu_0(k, c)) = \widebar{Y}_{\infty}(k, c),
    \end{aligned}
    \end{equation*}
    where $\widebar{Y}_{\infty}(k, c)$ denotes the average of all samples drawn from $(k, c)$. By the strong law of large numbers, since $N_{\infty}(k, c) = \infty$, $\widebar{Y}_{\infty}(k, c) = \lim_{n \rightarrow \infty} \widebar{Y}_n(k, c) = F(k, c)$ almost surely. It follows that $\lim_{n \rightarrow \infty} \mu_{n}(k, c) = F(k, c)$ almost surely.
\end{proof}

\end{document}